\documentclass[review]{elsarticle}
\usepackage{lineno,hyperref}
\usepackage{bbold}
\usepackage{verbatim}
\usepackage{url}
\usepackage{graphicx}
\usepackage{float}
\usepackage{mathrsfs}
\usepackage{epstopdf}
\usepackage[all,color]{xy}
\usepackage{cancel}
\usepackage{MnSymbol}
\usepackage{tikz}
\usepackage[margin=1in]{geometry}

\usetikzlibrary{arrows}
\modulolinenumbers[5]
\allowdisplaybreaks
\newtheorem{theorem}{Theorem}

\newtheorem{corollary}[theorem]{Corollary}
\newtheorem{example}[theorem]{Example}
\newtheorem{lemma}[theorem]{Lemma}
\newtheorem{proposition}[theorem]{Proposition}

\newtheorem{defn}[theorem]{Definition}
\newproof{proof}{Proof}
\numberwithin{equation}{subsection}
\numberwithin{theorem}{subsection}

\newcommand{\perm}{\mathrm{perm}}
\newcommand{\ignore}[1]{}

\newcommand{\cat}[1]{\mathfrak{#1}}

\DeclareMathOperator{\ob}{Ob}
\DeclareMathOperator{\card}{card}
\DeclareMathOperator{\inc}{inc}

\DeclareMathOperator{\Ran}{ran}
\DeclareMathOperator{\Pwr}{Pwr}
\DeclareMathOperator{\gen}{Gen}
\DeclareMathOperator{\sgn}{sgn}

\newcommand{\Set}{\mathbf{Set}}

\begin{document}

\begin{frontmatter}

\title{Incidence Hypergraphs: Injectivity, Uniformity, and Matrix-tree Theorems}

\author[add2]{Will Grilliette}
\author[add2]{Josephine Reynes\fnref{fn2}}
\author[add2]{Lucas J. Rusnak\corref{mycorrespondingauthor}}\ead{Lucas.Rusnak@txstate.edu}

\address[add2]{Department of Mathematics, Texas State University, San Marcos, TX 78666, USA}

\fntext[fn2]{Portions of these results appear in 2019 Honors thesis.}

\cortext[mycorrespondingauthor]{Corresponding author}

\begin{abstract}

An oriented hypergraph is an oriented incidence structure that allows for the generalization of graph theoretic concepts to integer matrices through its locally signed graphic substructure. The locally graphic behaviors are formalized in the subobject classifier of incidence hypergraphs.  Moreover, the injective envelope is calculated and shown to contain the class of uniform hypergraphs --- providing a combinatorial framework for the entries of incidence matrices. A multivariable all-minors characteristic polynomial is obtained for both the determinant and permanent of the oriented hypergraphic Laplacian and adjacency matrices arising from any integer incidence matrix. The coefficients of each polynomial are shown to be submonic maps from the same family into the injective envelope limited by the subobject classifier. These results provide a unifying theorem for oriented hypergraphic matrix-tree-type and Sachs-coefficient-type theorems. Finally, by specializing to bidirected graphs, the trivial subclasses for the degree-$k$ monomials of the Laplacian are shown to be in one-to-one correspondence with $k$-arborescences.
\end{abstract}


\end{frontmatter}

%%%%%%%%
\section{Introduction and Background}

\subsection{Introduction}

The Matrix-tree Theorem for graphs states that the first cofactor of the Laplacian matrix of a connected graph $G$ is equal to the tree-number of the graph, $\tau(G)$; this is equivalent to counting all rooted spanning trees for each $(v_i,v_i)$-minor to determine the coefficient of the degree-$1$ term of the characteristic polynomial of the Laplacian \cite{Tutte}. One common alternative presentation of the Matrix-tree Theorem are in terms of the product of the non-zero eigenvalues as they satisfy $\lvert V \rvert \cdot \tau(G) = \lambda_1 \lambda_2 \ldots \lambda_{n-1}$. Moreover, the $k^\text{th}$ minors count $k$-arborescences, and provide a combinatorial interpretation for the coefficients of the other monomials of the characteristic polynomials of the Laplacian --- 
with the second minors able to provide a combinatorial interpretation of Kirchhoff's Laws \cite{BSST,Tutte}. Sachs' Coefficient Theorem determines the coefficients of the characteristic polynomial of the adjacency matrix of a graph by counting reduced cycle-covers of the graph \cite{SGBook}, and has applications rooted in molecular orbitals \cite{SachsChem}. Generalizations of the Matrix-tree Theorem and Sachs' Coefficient Theorem to signed graphs appear in \cite{Seth1} and \cite{Sim1}, respectively. Signed graphs are a graph equipped with an edge signing function $E \rightarrow \{+1,-1\}$ and have their early roots in psychological balance and matroids \cite{Abel1, Har0, Har1, SG}. Orientations of signed graphs, or bidirected graphs, appear in \cite{MR0267898,OSG} and have applications in integer linear programs. Incidence orientations for hypergraphs allow for the study of integer matrices through their locally signed-graphic substructure \cite{IH1,AH1,OH1} and have applications in balanced matrices \cite{BM,DBM}. The spectral properties of oriented hypergraphs \cite{Reff6, Reff2, Reff5}, as well as their characteristic polynomials \cite{OHSachs, OHMTT} have received recent attention. 

We provide a unifying interpretation of the Matrix-tree Theorem and Sachs' Coefficient Theorem through the incidence structure of oriented hypergraphs while generalizing them to integer matrices and multivariate characteristic polynomials. This is accomplished by generalizing the work in \cite{OHSachs, OHMTT} and expanding on the category of incidence hypergraphs introduced in \cite{IH1} by identifying the subobject classifier (local graph behaviour) and injective envelope (uniform hypergraphs). The incidence structure is required to unify the study of the adjacency and Laplacian matrices as paths of half-integer length are crucial for the weak-walk (vertex-to-edge walks) interpretation introduced in \cite{OHHar,AH1}. An embedding of disjoint paths into the injective envelope is constructed to admit all possible embeddings to produce a hypergraphic generalization of cycle-covers to study integer matrices. We show that the coefficient of any given minor is the signed sum of restricted cycle-cover analogs that exists in $G$. Moreover, the trivial sub-classes of these maps correspond to $k$-arborescences when $G$ is a graph.
 
The classical Matrix-tree Theorem and Sachs' Coefficient Theorem are recalled in Subsection \ref{ssec:MTTandSach}, and an example is included to focus on the types of figures that will be unified by the weak-walk interpretation of the Laplacian and adjacency matrices for oriented hypergraphs. The background for incidence and oriented hypergraphs appear in Subsections \ref{ssec:IHbackground} and \ref{ssec:OHBackground}. Section \ref{sec:Cat} builds on the categorical foundation for incidence hypergraphs introduced in \cite{IH1}. For the category of incidence hypergraphs we characterize (1) the partial morphism representer, (2) the subobject classifier $\Omega_{\cat{R}}$, (3) subobjects, (4) the power-object, (5) injectivity, (6) essential monomorphisms, and (7) the injective envelope, where the injective envelope of an incidence-simple hypergraph is the minimal uniform hypergraph that contains it. Uniform oriented hypergraphs have been previously examined as alternatives to incidence duality and line graphs \cite{Reff3} and their connection to Hadamard matrices \cite{Reff5}.

Section \ref{sec:MTT} provides a characterization of oriented hypergraphic total-minor polynomials for both the permanent and determinant of the adjacency and Laplacian matrices. A family of incidence-path maps, introduced in \cite{OHSachs}, are used to form permutation clones that generalize the concept of a cycle-cover. The signed images of these maps can be interpreted as subobjects within the injective envelope, and the subobject classifier eliminates any that are not relevant to each minor calculation. In effect, this demonstrates that restrictions of uniform hypergraphic results can obtain all hypergraphic results, provided it takes place in the category of incidence hypergraphs. Tutte's $k$-arborescence Theorem is an immediate specialization of the these results as the single-element mapping families resulting from $k$ deletions in the pre-image are associated to each degree-$k$ monomial --- this is a strengthening of the results in \cite{OHMTT} which shows the single-element activation classes for each for degree-$1$ monomials are in one-to-one correspondence with Tutte's Matrix-Tree Theorem.

\subsection{The Matrix-tree and Sachs' Theorems}
\label{ssec:MTTandSach}

The Matrix-tree Theorem for graphs uses the determinant of the first minors of the graph Laplacian to count spanning trees \cite{Tutte}. It can be used to determine the coefficient of the degree-$1$ term of the characteristic polynomial of the Laplacian by counting all rooted spanning trees for each $(v_i,v_i)$-minor. 

\begin{theorem}[Matrix-tree Theorem]
\label{TMTT}
Let $G$ be a connected graph with Laplacian $\mathbf{L}_G$ and $ij$-minor $\mathbf{L}_{ij}$, then
\begin{align*}
\det(\mathbf{L}_{ij})=(-1)^{i+j}\tau(G).
\end{align*}
\end{theorem}

Larger minors count $k$-arborescences, with the second minors able to provide a combinatorial interpretation of Kirchhoff's Laws \cite{BSST}. 

\begin{example}
\label{ex:MTT1}
Consider $K_3$, the complete graph on $3$ vertices. It has the following Laplacian matrix and characteristic polynomial.
\begin{align*}
\mathbf{L}_{K_3}=\left[
\begin{array}{ccc}
2 & -1 & -1 \\ 
-1 & 2 & -1 \\ 
-1 & -1 & 2
\end{array}
\right]\\
\chi(\mathbf{L}_{K_3},x) = x^3 - 6x^2+ 9x 
\end{align*}

The coefficients of $\chi(\mathbf{L}_{K_3},x)$ are determined by $k$-arborescences --- a disjoint collected of $k$ rooted trees that span $G$. Figure \ref{fig:MTT1} depicts all the figures for each coefficient of $\chi(\mathbf{L}_{K_3},x)$.

\begin{figure}[!ht]
    \centering
    \includegraphics{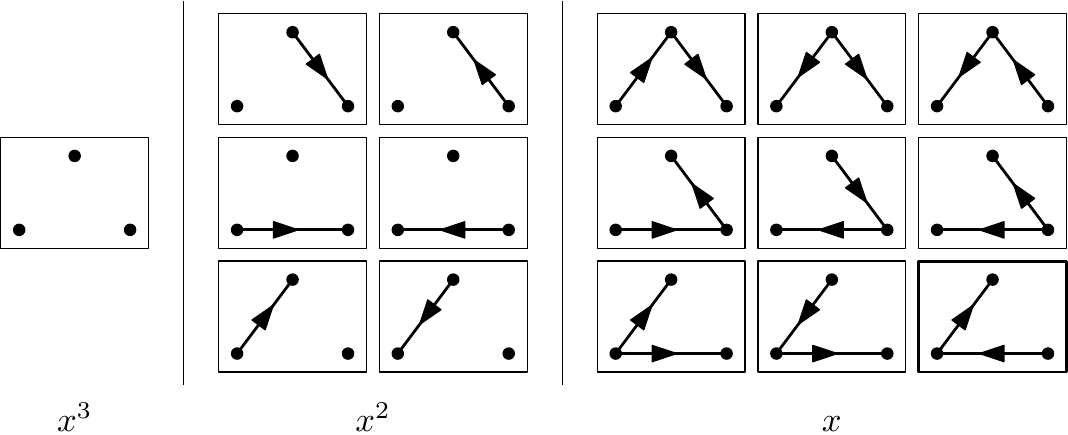}
    \caption{Rooted $k$-forests produce the coefficients of $\chi(\mathbf{L}_{K_3},x)$.}
    \label{fig:MTT1}
\end{figure}
\end{example}

Sachs' Theorem determines all the coefficients of the characteristic polynomial of the adjacency matrix of a graph by counting reduced cycle-covers of the graph \cite{SGBook}, and has its applications rooted in applications to molecular orbitals \cite{SachsChem}. A \emph{cycle-cover} of a graph $G$ is a disjoint collection of cycle-subgraphs, single edges, and isolated vertices. Let $\mathscr{U}_{k}$ be the set of all cycle covers of $G$ with exactly $k$ isolated vertices. Let $p(U)$ denote the number of components of in cycle cover $U$, and $c(U)$ denote the number of cycles in $U$.

\begin{theorem}[Sachs' Theorem]
\label{Sachs} Given a graph $G$ and adjacency matrix $\mathbf{A}_G$, the characteristic polynomial is
\begin{equation*}
\chi(\mathbf{A}_G,x)=\sum\limits_{k=1}^{\lvert V(G) \rvert}\left( \sum\limits_{U\in \mathscr{U}%
_{k}}(-1)^{p(U)}(2)^{c(U)}\right) x^{k}\text{.}
\end{equation*}
\end{theorem}

\begin{example}
\label{ex:Sachs1}
Again, consider the complete graph on $3$ vertices $K_3$. It has the following adjacency matrix and characteristic polynomial.
\begin{align*}
\mathbf{A}_{K_3}=\left[
\begin{array}{ccc}
0 & 1 & 1 \\ 
1 & 0 & 1 \\ 
1 & 1 & 0
\end{array}
\right]\\
\chi(\mathbf{A}_{K_3}, x) = x^3 - 3x - 2 
\end{align*}

The coefficients of $\chi(\mathbf{A}_{K_3}, x)$ are determined by cycle-covers. Figure \ref{fig:Sachs1} depicts all the figures for each coefficient of $\chi(\mathbf{A}_{K_3}, x)$.

\begin{figure}[!ht]
    \centering
    \includegraphics{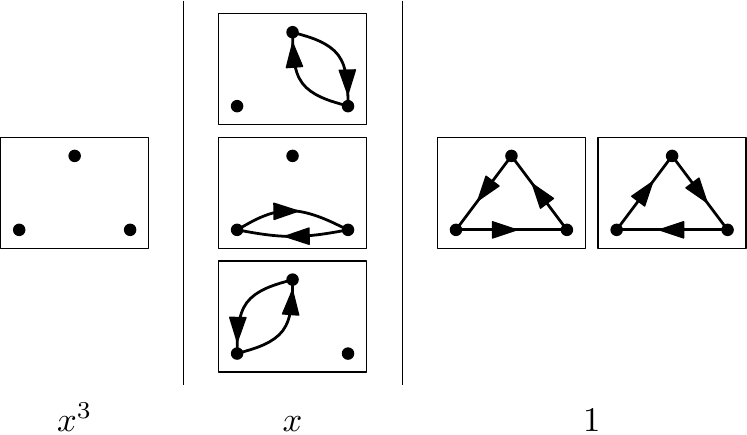}
    \caption{Reduced cycle-covers with $k$ isolated vertices.}
    \label{fig:Sachs1}
\end{figure}

The $3$-cycles are grouped along the reorientation count $(2)^{c(U)}$ calculation. However, the $2$-cycles are actually a closed walk along a single edge and are equivalent to their own reversal, hence are only counted once each.
\end{example}

\subsection{Incidence Hypergraphs}
\label{ssec:IHbackground}

The categorical foundation for incidence theory introduced in \cite{IH1} demonstrated the deficiencies of other graph-like categories \cite{Berge2,dorfler1980, GandH}, culminating with a characterization of (directed) graph exponentials as homomorphisms in the category of incidence structures. A categorical description is critical for study of (hyper)graph homomorphisms as the weak-walk interpretation of the adjacency and Laplacian matrices \cite{AH1,OHHar} is a path homomorphism theorem which derives the combinatorial aspects of the adjacency and Laplacian matrices. Moreover, traditional graph-theoretic or set-system hypergraphic approaches generally fail due to a lack of exponential objects, subobject classifiers, and injective envelopes that do not represent matrix algebra; the injective envelopes for graphs and set-systems are complete graphs and simplicial sets, respectively \cite{IH1, Grill2, Grill1}. 

An incidence hypergraph is a quintuple $G=(\check{V}, \check{E}, I, \varsigma, \omega)$ consisting of a set of vertices $\check{V}$, a set of edges $\check{E}$, a set of incidences $I$, and two incidence maps $\varsigma:I\to\check{V}$ and $\omega:I\to\check{E}$. Note, this notation is from \cite{IH1}, where it was demonstrated that the ``vertex set'' is a functor into $\Set$ relative to each category. The set decorations distinguish between such functors; for example, $\check{V}(G)$ is the set of vertices of an incidence hypergraph, while $V(G)$ is the set of vertices of a graph. While there is no appreciable difference on the objects, the differences in the morphisms and adjoints are in \cite{IH1}.

\begin{figure}[!ht]
    \centering
    \includegraphics{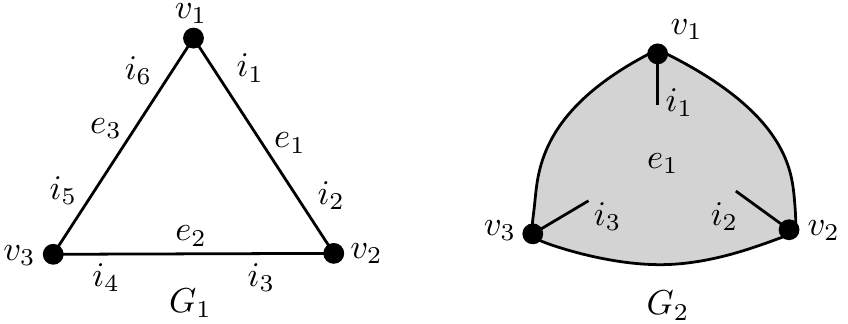}
    \caption{Example objects in the category of incidence hypergraphs: a $K_3$ graph regarded as an incidence structure and a single $3$-edge.}
    \label{fig:K3Graph}
\end{figure}

Formally, an incidence hypergraph (from \cite[p.\ 17]{IH1}) is defined as follows:
Let $\cat{D}$ be the finite category
\[\xymatrix{
0	&	2\ar[l]_{y}\ar[r]^{z}	&	1\\
}\]
and the category of incidence hypergraphs is $\cat{R}:=\Set^{\cat{D}}$ with evaluation functors $\xymatrix{\Set & \cat{R}\ar@/_/[l]_(0.4){\check{V}}\ar@/^/[l]^(0.4){\check{E}}\ar[r]^(0.4){I} & \Set}$ at $0$, $1$, and $2$, respectively.  As discussed in \cite{IH1}, the category of incidence hypergraphs $\cat{R}$ is most closely related to the category of quivers $\cat{Q}$, as there is an atomic geometric surjection between them. An object $G$ of $\cat{R}$ consists of the following:  a set $\check{V}(G)$, a set $\check{E}(G)$, a set $I(G)$, a function $\varsigma_{G}:I(G)\to\check{V}(G)$, and a function $\omega_{G}:I(G)\to\check{E}(G)$. Note that the standard $(v,e)$ incidence function $\iota_{G}: I(G)\rightarrow\check{V}(G)\times\check{E}(G)$ used in \cite{OHSachs,OHMTT} is the unique factorization of $\varsigma_{G}$ and $\omega_{G}$ through the cartesian product.

A \emph{directed path of length $n/2$} is a non-repeating sequence 
\begin{equation*}
\overrightarrow{P}_{n/2}=(a_{0},i_{1},a_{1},i_{2},a_{2},i_{3},a_{3},...,a_{n-1},i_{n},a_{n})
\end{equation*}
of vertices, edges, and incidences, where $\{a_{\ell }\}$ is an alternating sequence of vertices and edges, and $i_{j}$ is an incidence between $a_{j-1}$ and $a_{j}$. The \emph{tail} of a path is $a_0$ and the \emph{head} of a path is $a_n$. In terms of paths, the generators of $\cat{R}$ are the path of length zero consisting of a single vertex, the path of length zero consisting of a single edge, and the $1$-edge. The $1$-edge generator is critical to the structure theorems as it is also terminal, and allows for the formation of weak walks.

A \emph{directed weak walk of $G$} is the image of an incidence-preserving map of a directed path into $G$. A \emph{backstep of }$G$ is a non-incidence-monic map of $\overrightarrow{P}_{1}$ into $G$; a \emph{loop of }$G$ is an incidence-monic map of $\overrightarrow{P}_{1}$ into $G$ that is not vertex-monic; and a \emph{directed adjacency of }$G$ is a map of $ \overrightarrow{P}_{1}$ into $G$ that is incidence-monic. Observe that loops are considered adjacencies while backsteps are not, and can respectively be regarded as orientable and non-orientable $1$-cycles. A \emph{contributor of $G$} is an incidence preserving map from a disjoint union of $\overrightarrow{P}_{1}$'s with tail $t$ and head $h$ into $G$ defined by $c:\dcoprod \limits_{v\in V}\overrightarrow{P}_{1}\rightarrow G$ such that $c(t_{v})=v$ and $\{c(h_{v})\mid v\in V\}=V$. A contributor can be regarded as a permutation clone that lives in the incidence hypergraph, with backsteps and loops playing the role of fixed points. Since there are often many contributors whose head-tail mapping produce the same permutation, two contributors that correspond to the same permutation are called \emph{permutomorphic}. Let $\mathcal{C}(G)$ denote the set of contributors of $G$. A \emph{strong contributor} is an incidence-monic contributor, and let $\mathcal{S}(G)$ denote the set of strong contributors. Two contributors are said to be \emph{tail-equivalent} is their tail images agree.

\begin{example}
Figure \ref{Fig:Hypergraph3} depicts the set of all contributors for both $G_1$ and $G_2$ from Figure \ref{fig:K3Graph}. Backsteps are depicted as a half-edge whose incidence is traversed twice. The backstep-free contributors are the strong contributors, and are analogous to cycle-covers for the adjacency matrix. However, the set of all contributors are analogous to cycle-covers for the Laplacian matrix.

\begin{figure}[H]
    \centering
    \includegraphics{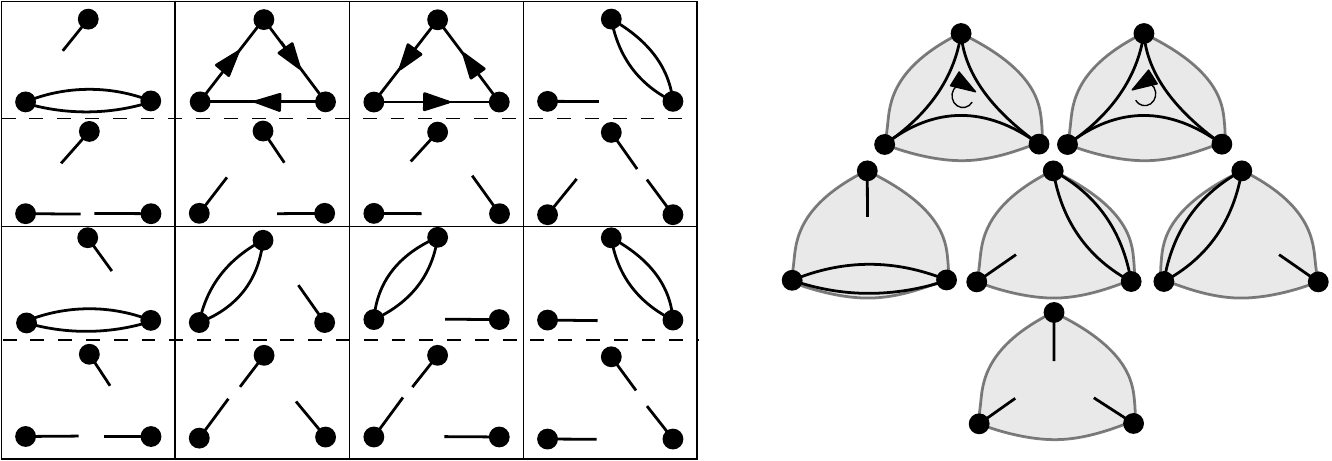}
    \caption{Contributors of $G_1$ and $G_2$ from Figure \ref{fig:K3Graph}, grouped by their tail-equivalent families.}
    \label{Fig:Hypergraph3}
\end{figure}

The contributors are collected into their tail-equivalent classes, with the identity permutation clone as the bottom element (containing only backsteps). The grouped elements are separated by a dashed line in the left figure, while the right figure only has one class.
\end{example}

\subsection{Oriented Hypergraphs}
\label{ssec:OHBackground}

An \emph{orientation of an incidence hypergraph G} is a signing function $\sigma:I\rightarrow\{+1,-1\}$. The \emph{sign of a weak walk} $W$ is 
\begin{equation*}
\sgn(W)=(-1)^{\lfloor n/2\rfloor }\prod_{h=1}^{n}\sigma (i_{h})\text{,}
\end{equation*}
which is equivalent to taking the product of the signed adjacencies if $W$ is a vertex-walk. Extroverted/introverted adjacencies are negative while two incidences that compatibly traverse an adjacency are positive; see \cite{MR0267898, SG, OSG} for bidirected graphs as orientations of signed graphs.

\begin{figure}[ht!]
    \centering
    \includegraphics{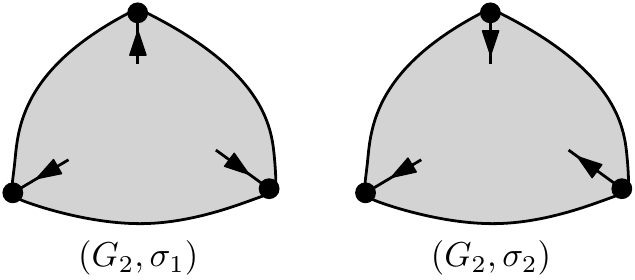}
    \caption{Two incidence orientations of the $3$-edge graph $G_2$ from Figure \ref{fig:K3Graph}, with $\sigma_1$ having all adjacencies negative.}
    \label{fig:2Orientations}
\end{figure}

The \emph{incidence matrix} of an oriented hypergraph $G$ is the $V \times E$ matrix $\mathbf{H}_{G}$ where the $(v,e)$-entry is the sum of $\sigma
(i)$ for each $i \in I$ such that $\varsigma (i)=v$ and $\omega (i)=e$. The \emph{adjacency matrix} $\mathbf{A}_{G} $ of an oriented hypergraph $G$ is the $V\times V$ matrix whose $(u,w)$-entry is the sum of $\sgn(q(\overrightarrow{P}_{1}))$ for all incidence monic maps $q:\overrightarrow{P}_{1}\rightarrow G$ with $q(\varsigma(i_1))=u$ and $q(\varsigma(i_2))=w$. The \emph{%
degree matrix} of an oriented hypergraph $G$ is the $V\times V$ diagonal
matrix whose $(v,v)$ -entry is the sum of all non-incidence-monic maps $p:\overrightarrow{P}_{1}\rightarrow G$ with $p(\varsigma(i_1))=p(\varsigma(i_2)=v$.
The \emph{Laplacian matrix of $G$} is defined as $\mathbf{L}_{G}:=\mathbf{H}_{G} \mathbf{H}_{G}^{T}=\mathbf{D}_{G}-\mathbf{A}_{G}$ for all oriented hypergraphs, see \cite{AH1} for the result that the Laplacian is the 1-weak-walk matrix.

The Laplacians of the two oriented hypergraphs in Figure \ref{fig:2Orientations} are
$$\mathbf{L}_{(G_2,\sigma_1)}=\left[
\begin{array}{ccc}
1 & 1 & 1 \\ 
1 & 1 & 1 \\ 
1 & 1 & 1 %
\end{array}
\right], \hspace{1em}
\mathbf{L}_{(G_2,\sigma_2)}=\left[
\begin{array}{ccc}
1 & 1 & -1 \\ 
1 & 1 & -1 \\ 
-1 & -1 & 1 %
\end{array}
\right],$$
where introverted/extroverted adjacencies are regarded as negative, and $\sigma_1$ corresponding the the signless Laplacian. Since incidence hypergraphs can be regarded as an oriented hypergraph with a constant orientation functions, incidence hypergraphs alone naturally model the signless Laplacian, see \cite{OHMTT}.

Generalizations of Sachs' Theorem and the permanental polynomial to signed graphs appear in \cite{Sim1}, and Theorem \ref{OHSachsT} below is from \cite[Theorem 4.2.1]{OHSachs} and generalizes these results to oriented hypergraphs and integer matrices. Let $\chi ^{D}(\mathbf{M},x):=\det (x\mathbf{I-M)}$ be the determinant-based
characteristic polynomial and $\chi ^{P}(\mathbf{M},x):=\mathrm{perm}(x%
\mathbf{I-M)}$ be the permanent-based characteristic polynomial. 
\begin{theorem}[\cite{OHSachs}, Theorem 4.2.1]
\label{OHSachsT}
Let $G$ be an oriented hypergraph with adjacency matrix $\mathbf{A}_{G}$ and Laplacian matrix $\mathbf{L}_{G}$, then

\begin{enumerate}
\item $\chi ^{P}(\mathbf{A}_{G},x)=\dsum\limits_{k=0}^{\left\vert
V\right\vert }\left( \dsum\limits_{c\in \hat{\mathcal{C}}_{=k}(G)}(-1)^{oc(c)+nc(c)}\right) x^{k}$,

\item $\chi ^{D}(\mathbf{A}_{G},x)=\dsum\limits_{k=0}^{\left\vert
V\right\vert }\left( \dsum\limits_{c\in \hat{\mathcal{C}}%
_{=k}(G)}(-1)^{pc(c)}\right) x^{k}$,

\item $\chi ^{P}(\mathbf{L}_{G},x)=\dsum\limits_{k=0}^{\left\vert
V\right\vert }\left( \dsum\limits_{c\in \hat{\mathcal{C}}_{\geq
k}(G)}(-1)^{nc(c)+bs(c)}\right) x^{k}$,

\item $\chi ^{D}(\mathbf{L}_{G},x)=\dsum\limits_{k=0}^{\left\vert
V\right\vert }\left( \dsum\limits_{c\in \hat{\mathcal{C}}_{\geq
k}(G)}(-1)^{ec(c)+nc(c)+bs(c)}\right) x^{k}$.
\end{enumerate}
Where $bs(c)$ is the number of backsteps in contributor $c$, $oc(c)/ ec(c)/nc(c)/pc(c)$ is the number of odd/even/posititve/negative circles in $c$, $\hat{\mathcal{C}}_{=k}(G)$ is the set of contributors with exactly $k$ backsteps and $k$ removed, and $\hat{\mathcal{C}}_{\geq k}(G)$ is the set of contributors with $k$ or more backsteps and $k$ removed.
\end{theorem}

We improve upon this theorem and prove a multivariate all-minor generalization that unifies Sachs' theorem and the Matrix-tree theorem to integer incidence matrices through the locally signed graphic contributors of their associated oriented hypergraph. Moreover, we exhibit that these types of theorems are a result of the category of incidence hypergraphs being a topos and tied intimately to the subobject classifier and the injective envelope --- leaving open the possibility of having a purely algebraic formulation of matrix-tree-like theorems.

%%%%%%%%%%%
%%%%%%%%%%%
\section{Subobjects \& Injective Envelopes}
\label{sec:Cat}

\subsection{Partial Morphism Representer}

As a presheaf topos, $\cat{R}$ possesses many well-documented categorical properties:  completeness, cocompleteness \cite[Corollary I.2.15.4]{borceux}, a subobject classifier \cite[Lemma A1.6.6]{elephant1}, and partial morphism representers \cite[Proposition A2.4.7]{elephant1}.  Most important for this discussion will be the partial morphism representer, which will be used to characterize injective objects and the injective envelope.  Since the partial morphism representer is a generalization of the subobject classifier, both will be developed together.

The subobject classifier $\Omega_{\cat{R}}$ acts like a generalization of the 2-element set $\{0,1\}$, where 1 serves as ``true'' and 0 as ``false''.  General constructions involve sieves \cite[p.\ 37-39]{maclane2} or subfunctors \cite[Example III.5.2.5]{borceux}, but the following construction will be done concretely with sets.  The first component of this structure is the terminal object of the category, which will generates the ``true'' components for the subobject classifier.  Using the adjoints of $I$ \cite[p.\ 18]{IH1}, both the terminal and initial object can be identified immediately.

\begin{defn}[Initial \& terminal]
Let $\mathbb{0}_{\cat{R}}:=I^{\diamond}(\emptyset)$ and $\mathbb{1}_{\cat{R}}:=I^{\star}\left(\{1\}\right)$.  As $I^{\diamond}$ is cocontinuous and $\emptyset$ is initial in $\Set$, $\mathbb{0}_{\cat{R}}$ is initial in $\cat{R}$.  As $I^{\star}$ is continuous and $\{1\}$ is terminal in $\Set$, $\mathbb{1}_{\cat{R}}$ is terminal in $\cat{R}$.
\end{defn}

Next, the partial morphism representer of an incidence hypergraph $G$ will be built by extending the structure of $G$.  The original structure of $G$ will serve as the ``true'' values, while the extended structure will serve as the ``false'' values:  a new vertex, a new edge, and new incidences between every vertex and edge.  This process is functorial, and application of this process to $\mathbb{1}_{\cat{R}}$ produces the ``false'' components for the subobject classifier.

\begin{defn}[Partial morphism representer construction]
For an incidence hypergraph $G$, define an incidence hypergraph $\tilde{G}$ by
\begin{enumerate}
\item $\check{V}\left(\tilde{G}\right):=\left(\{1\}\times\check{V}(G)\right)\cup\left\{(0,0)\right\}$, $\check{E}\left(\tilde{G}\right):=\left(\{1\}\times\check{E}(G)\right)\cup\left\{(0,0)\right\}$,
\item $I\left(\tilde{G}\right):=\left(\{1\}\times I(G)\right)\cup\left(\{0\}\times\check{V}\left(\tilde{G}\right)\times\check{E}\left(\tilde{G}\right)\right)$,
\item $\varsigma_{\tilde{G}}(a):=\left\{\begin{array}{cc}
\left(1,\varsigma_{G}(i)\right),	&	a=(1,i),\\
v,	&	a=(0,v,e),\\
\end{array}\right.$
$\omega_{\tilde{G}}(a):=\left\{\begin{array}{cc}
\left(1,\omega_{G}(i)\right),	&	a=(1,i),\\
e,	&	a=(0,v,e).\\
\end{array}\right.$
\end{enumerate}
Define the incidence hypergraph homomorphism $\xymatrix{G\ar[r]^{\eta_{G}} & \tilde{G}}$ by $\check{V}\left(\eta_{G}\right)(v):=(1,v)$, $\check{E}\left(\eta_{G}\right)(e):=(1,e)$, and $I\left(\eta_{G}\right)(i):=(1,i)$.  Note that $\eta_{G}$ is monic by \cite[Corollary I.2.15.3]{borceux}.
\end{defn}

\begin{theorem}[Partial morphism representer characterization]
If $\xymatrix{K & \textrm{ }H\ar@{>->}[l]_{\phi}\ar[r]^{\psi} & G}$ are incidence hypergraph homomorphisms where $\phi$ is monic, there is a unique incidence hypergraph homomorphism $\xymatrix{K\ar[r]^{\hat{\psi}} & \tilde{G}}$ such that $\xymatrix{K & \textrm{ }H\ar@{>->}[l]_{\phi}\ar[r]^{\psi} & G}$ is a pullback of $\xymatrix{K\ar[r]^{\hat{\psi}} & \tilde{G} & \textrm{ }G\ar@{>->}[l]_(0.4){\eta_{G}}}$.  Consequently, $\tilde{G}$ equipped with $\eta_{G}$ is a partial morphism representer of $G$.
\end{theorem}

\begin{proof}

Define $\xymatrix{K\ar[r]^{\hat{\psi}} & \tilde{G}}\in\cat{R}$ by
\begin{itemize}
\item $\check{V}\left(\hat{\psi}\right)(v):=\left\{\begin{array}{cc}
\left(1,\check{V}(\psi)(w)\right),	&	v=\check{V}(\phi)(w),\\
\left(0,0\right),	&	\textrm{otherwise},\\
\end{array}\right.$
$\check{E}\left(\hat{\psi}\right)(e):=\left\{\begin{array}{cc}
\left(1,\check{E}(\psi)(f)\right),	&	e=\check{E}(\phi)(f),\\
\left(0,0\right),	&	\textrm{otherwise},\\
\end{array}\right.$
\item $I\left(\hat{\psi}\right)(i):=\left\{\begin{array}{cc}
\left(1,I(\psi)(j)\right),	&	i=I(\phi)(j),\\
\left(0,\check{V}\left(\hat{\psi}\right)\left(\varsigma_{K}(i)\right),\check{E}\left(\hat{\psi}\right)\left(\omega_{K}(i)\right)\right),	&	\textrm{otherwise}.\\
\end{array}\right.$
\end{itemize}
As $\phi$ is monic, $\hat{\psi}$ is well-defined.  Routine checks show the pullback condition and uniqueness of $\hat{\psi}$.\qed

\end{proof}

\begin{corollary}[Functor $\tilde{\Box}$]
If $\xymatrix{G\ar[r]^{\phi} & H}$ is an incidence hypergraph homomorphism, then the incidence hypergraph homomorphism $\xymatrix{\tilde{G}\ar[r]^{\tilde{\phi}} & \tilde{H}}$ is given by
\begin{enumerate}
\item $\check{V}\left(\tilde{\phi}\right)(n,x):=\left\{\begin{array}{cc}
\left(1,\check{V}(\phi)(x)\right),	&	n=1,\\
(0,0),	&	n=0,\\
\end{array}\right.$
$\check{E}\left(\tilde{\phi}\right)(n,y):=\left\{\begin{array}{cc}
\left(1,\check{E}(\phi)(y)\right),	&	n=1,\\
(0,0),	&	n=0,\\
\end{array}\right.$
\item $I\left(\tilde{\phi}\right)(n,z):=\left\{\begin{array}{cc}
\left(1,I(\phi)(z)\right),	&	n=1,\\
\left(0,\check{V}\left(\tilde{\phi}\right)\left(\varsigma_{\tilde{G}}(z)\right),\check{E}\left(\tilde{\phi}\right)\left(\omega_{\tilde{G}}(z)\right)\right),	&	n=0.\\
\end{array}\right.$
\end{enumerate}
\end{corollary}

\begin{corollary}[Subobject classifier $\Omega_{\cat{R}}$]
The incidence hypergraph $\Omega_{\cat{R}}:=\tilde{\mathbb{1}}_{\cat{R}}$ equipped with $t_{\cat{R}}:=\eta_{\mathbb{1}_{\cat{R}}}$ is a subobject classifier for $\cat{R}$.
\begin{figure}[ht!]
    \centering
    \includegraphics{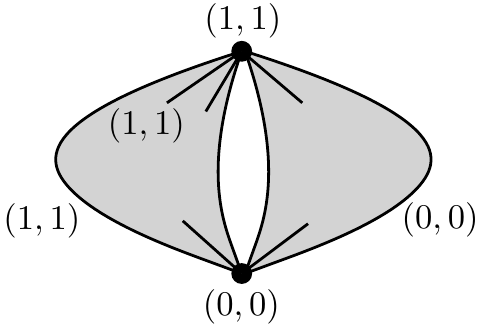}
    \caption{The subobject classifier $\Omega_{\cat{R}}$, with all variants of ``true'' and ``false'' labelled.}
    \label{fig:SubObClass}
\end{figure}
\end{corollary}

Please note that $\{0,1\}$ gives a binary choice between ``true'' and ``false'' for $\Set$, but $\Omega_{\cat{R}}$ has many different types of ``true'' and ``false'' for $\cat{R}$.  There is a unique ``true'' vertex, edge, and incidence indictated by the 1's above.  However, while there is a ``false'' vertex and a ``false'' edge, there are four distinct ``false'' incidences.

\subsection{Subobjects}

This section takes a moment to develop the properties of $\Omega_{\cat{R}}$ further before proceeding to injectivity.  As its name implies, the subobject classifier $\Omega_{\cat{R}}$ characterizes substructures of an incidence hypergraph $G$ via homomorphisms from $G$ into $\Omega_{\cat{R}}$ or, equivalently, as global elements of the exponential ${\Omega_{\cat{R}}}^{G}$.  For an incidence hypergraph, a ``global element'' translates to an incidence of the hypergraph.

\begin{lemma}[Global elements]
For an incidence hypergraph $G$, the global elements of $G$ correspond to the elements of $I(G)$.
\end{lemma}

\begin{proof}
From \cite[p.\ 18]{IH1}, note that $\mathbb{1}_{\cat{R}}=I^{\star}\left(\{1\}\right)=I^{\diamond}\left(\{1\}\right)$.  Thus,
\[
\cat{R}\left(\mathbb{1}_{\cat{R}},G\right)
=\cat{R}\left(I^{\diamond}\left(\{1\}\right),G\right)
\cong\Set\left(\{1\},I(G)\right)
\cong I(G).
\]\qed
\end{proof}

On the other hand, one has the natural notion of ``subhypergraph,'' which is defined formally in accordance with ``subgraph'' and ``subdigraph'' \cite[Definitions 1.3.1 \& 2.2.3.1]{balakrishnan}.

\begin{defn}[Subhypergraph]
Given an incidence hypergraph $G$, a \emph{subhypergraph} of $G$ is an incidence hypergraph $K$ such that the following conditions hold:
\begin{enumerate}
\item $\check{V}(K)\subseteq\check{V}(G)$,
$\check{E}(K)\subseteq\check{E}(G)$,
$I(K)\subseteq I(G)$,
\item $\varsigma_{K}(i)=\varsigma_{G}(i)$, $\omega_{K}(i)=\omega_{G}(i)$ for all $i\in I(K)$.
\end{enumerate}
The \emph{canonical inclusion} $\xymatrix{K\ar[r]^{\iota_{K}} & G}$ is the incidence hypergraph homomorphism such that $\check{V}\left(\iota_{K}\right)$, $\check{E}\left(\iota_{K}\right)$, and $I\left(\iota_{K}\right)$ are the set-theoretic inclusions.
\end{defn}

In order to connect subhypergraphs concretely to maps into $\Omega_{\cat{R}}$, the following notion of generation is lifted from abstract algebra.  Recall that a subgroup can be generated from a collection of elements within a group.  For incidence hypergraphs, one can generate the least subhypergraph containing a collection of components within an existing incidence hypergraph.

\begin{defn}[Generated subhypergraph]
Given an incidence hypergraph $G$, let $S_{1}\subseteq\check{V}(G)$, $S_{2}\subseteq\check{E}(G)$, and $S_{3}\subseteq I(G)$.  Let $\xymatrix{S_{1}\ar[r]^(0.4){j_{1}} & \check{V}(G)},\xymatrix{S_{2}\ar[r]^(0.4){j_{2}} & \check{E}(G)}$, $\xymatrix{S_{3}\ar[r]^(0.4){j_{3}} & I(G)}$ be the set-theoretic inclusion maps.  There are unique incidence hypergraph homomorphisms $\xymatrix{\check{V}^{\diamond}\left(S_{1}\right)\ar[r]^(0.6){\hat{j}_{1}} & G},\xymatrix{\check{E}^{\diamond}\left(S_{2}\right)\ar[r]^(0.6){\hat{j}_{2}} & G},\xymatrix{I^{\diamond}\left(S_{3}\right)\ar[r]^(0.6){\hat{j}_{3}} & G}$ such that $\check{V}\left(\hat{j}_{1}\right)=j_{1}$, $\check{E}\left(\hat{j}_{2}\right)=j_{2}$, $I\left(\hat{j}_{3}\right)=j_{3}$.  Let $\varpi_{n}$ be the canonical inclusions into $\check{V}^{\diamond}\left(S_{1}\right)\coprod\check{E}^{\diamond}\left(S_{2}\right)\coprod I^{\diamond}\left(S_{3}\right)$ for $n=1,2,3$.  There is a unique incidence hypergraph homomorphism $\xymatrix{\check{V}^{\diamond}\left(S_{1}\right)\coprod\check{E}^{\diamond}\left(S_{2}\right)\coprod I^{\diamond}\left(S_{3}\right)\ar[r]^(0.8){\phi} & G}$ such that $\phi\circ\varpi_{n}=\hat{j}_{n}$ for $n=1,2,3$.  Define the subhypergraph $\gen_{G}\left(S_{1},S_{2},S_{3}\right)$ of $G$ by
\begin{enumerate}
\item $\check{V}\gen_{G}\left(S_{1},S_{2},S_{3}\right):=\Ran\left(\check{V}(\phi)\right)$,
\item $\check{E}\gen_{G}\left(S_{1},S_{2},S_{3}\right):=\Ran\left(\check{E}(\phi)\right)$,
\item $I\gen_{G}\left(S_{1},S_{2},S_{3}\right):=\Ran\left(I(\phi)\right)$.
\end{enumerate}
\end{defn}

When generating a subgroup from a subset of a group, more elements arise as products of the generating elements and their inverses.  Likewise for incidence hypergraphs, an incidence used for generation of a subhypergraph forces its corresponding vertex and edge to appear.

\begin{proposition}[Structure of generated subhypergraph]
Given an incidence hypergraph $G$, let $S_{1}\subseteq\check{V}(G)$, $S_{2}\subseteq\check{E}(G)$, and $S_{3}\subseteq I(G)$.  Then, one has
\begin{enumerate}
\item $\check{V}\gen_{G}\left(S_{1},S_{2},S_{3}\right)=S_{1}\cup\mathcal{P}\left(\varsigma_{G}\right)\left(S_{3}\right)$,
\item $\check{E}\gen_{G}\left(S_{1},S_{2},S_{3}\right)=S_{2}\cup\mathcal{P}\left(\omega_{G}\right)\left(S_{3}\right)$,
\item $I\gen_{G}\left(S_{1},S_{2},S_{3}\right)=S_{3}$.
\end{enumerate}
\end{proposition}

\begin{proof}
Peeling away the universal constructions, the vertex set arises from the following calculation.
\[\begin{array}{rcl}
\check{V}\gen_{G}\left(S_{1},S_{2},S_{3}\right)
&   =   &   \Ran\left(\check{V}(\phi)\right)
=\mathcal{P}\check{V}(\phi)\left(\{1\}\times S_{1}\right)\cup\emptyset\cup\mathcal{P}\check{V}(\phi)\left(\{3\}\times S_{3}\right)\\
&   =   &   \mathcal{P}\check{V}\left(\phi\circ\varpi_{1}\right)\left(S_{1}\right)\cup\mathcal{P}\check{V}\left(\phi\circ\varpi_{3}\right)\left(S_{3}\right)
=\mathcal{P}\check{V}\left(\hat{j}_{1}\right)\left(S_{1}\right)\cup\mathcal{P}\check{V}\left(\hat{k}\right)\left(S_{3}\right)\\
&   =   &   \mathcal{P}\left(j_{1}\right)\left(S_{1}\right)\cup\mathcal{P}\check{V}\left(\hat{k}\right)\left(S_{3}\right)
=S_{1}\cup\mathcal{P}\left(\check{V}\left(\hat{k}\right)\circ\varsigma_{I^{\diamond}\left(S_{3}\right)}\right)\left(S_{3}\right)
=S_{1}\cup\mathcal{P}\left(\varsigma_{G}\circ I\left(\hat{k}\right)\right)\left(S_{3}\right)\\
&   =   &   S_{1}\cup\mathcal{P}\left(\varsigma_{G}\circ k\right)\left(S_{3}\right)
=S_{1}\cup\mathcal{P}\left(\varsigma_{G}\right)\left(S_{3}\right)\\
\end{array}\]
Similar calculations yield the edge and incidence sets.\qed
\end{proof}

With this notion of generation in hand, the intuitional notion of ``subhypergraph'' captures the subobjects in $\cat{R}$ via the unique characteristic map into $\Omega_{\cat{R}}$.

\begin{theorem}[Subobject characterization]\label{subobject}
For incidence hypergraph $G$, the subobjects of $G$ correspond precisely to subhypergraphs of $G$.
\end{theorem}

\begin{proof}

By \cite[Proposition III.5.1.6]{borceux}, the subobjects of $G$ correspond bijectively to the elements of the following set:
\[
I\left({\Omega_{\cat{R}}}^{G}\right)
\cong\cat{R}\left(\mathbb{1}_{\cat{R}},{\Omega_{\cat{R}}}^{G}\right)
\cong\cat{R}\left(G\prod\mathbb{1}_{\cat{R}},\Omega_{\cat{R}}\right)
\cong\cat{R}\left(G,\Omega_{\cat{R}}\right).
\]
Let $\mathscr{S}:=\left\{K\in\ob(\cat{R}):K\textrm{ is a subhypergraph of }G\right\}$.  Given $K\in\mathscr{S}$, then $\iota_{K}$ is monic, so there is a unique $\xymatrix{G\ar[r]^{\chi_{K}} & \Omega_{\cat{R}}}\in\cat{R}$ such that $\xymatrix{G & K\ar[r]^{\mathbf{1}_{K}}\ar[l]_{\iota_{K}} & \mathbb{1}_{\cat{R}}}$ is a pullback of $\xymatrix{G\ar[r]^{\chi_{K}} & \Omega_{\cat{R}} & \mathbb{1}_{\cat{R}}\ar[l]_{t_{\cat{R}}}}$.  Define $\Phi:\mathscr{S}\to\cat{R}\left(G,\Omega_{\cat{R}}\right)$ by $\Phi(K):=\chi_{K}$.

Say $K,L\in\mathscr{S}$ satisfy that $\Phi(K)=\Phi(L)$.  Then, $\chi_{K}=\chi_{L}$, so both $\xymatrix{G & K\ar[r]^{\mathbf{1}_{K}}\ar[l]_{\iota_{K}} & \mathbb{1}_{\cat{R}}}$ and $\xymatrix{G & L\ar[r]^{\mathbf{1}_{L}}\ar[l]_{\iota_{L}} & \mathbb{1}_{\cat{R}}}$ are pullbacks of $\xymatrix{G\ar[rr]^{\chi_{K}=\chi_{L}} & & \Omega_{\cat{R}} & \mathbb{1}_{\cat{R}}\ar[l]_{t_{\cat{R}}}}$.  There is a unique isomorphism $\xymatrix{L\ar[r]^{\alpha} & K}\in\cat{R}$ such that $\iota_{K}\circ\alpha=\iota_{L}$ and $\mathbf{1}_{K}\circ\alpha=\mathbf{1}_{L}$.  For $v\in\check{V}(L)$, one has
\[
v
=\check{V}\left(\iota_{L}\right)(v)
=\check{V}\left(\iota_{K}\right)\left(\check{V}(\alpha)(v)\right)
=\check{V}(\alpha)(v)
\in\check{V}(K),
\]
showing $\check{V}(L)\subseteq\check{V}(K)$.  A dual argument shows equality.  Likewise, one has $\check{E}(L)=\check{E}(K)$ and $I(L)=I(K)$, giving $L=K$.

Let $\chi\in\cat{R}\left(G,\Omega_{\cat{R}}\right)$.  Define $K:=\gen_{G}\left(\check{V}(\chi)^{-1}(1,1),\check{E}(\chi)^{-1}(1,1),I(\chi)^{-1}(1,1)\right)$.  Then, $K\in\mathscr{S}$, and a calculation shows that $\Phi(K)=\chi$.\qed

\end{proof}

To conclude this section, a concrete representation of the ``power hypergraph'' $\Pwr(G):={\Omega_{\cat{R}}}^{G}$ is established.  While the power hypergraph can be represented in terms of homomorphisms \cite[Definition 3.43]{IH1}, the following representation immediately and intuitively connects to notions of power objects and subobjects.  Observe that this power hypergraph is contravariant and is deeply connected to the preimage operation of sets.  Moreover, the ``element'' map of the adjunction encodes membership for subhypergraphs using the different notions of ``true'' and ``false'' discussed in the previous section.

\begin{defn}[Power hypergraph]
Given an incidence hypergraph $G$, define the incidence hypergraph $\Pwr(G)$ by
\begin{enumerate}
\item $\check{V}\Pwr(G):=\mathcal{P}\check{V}(G)$, $\check{E}\Pwr(G):=\mathcal{P}\check{E}(G)$,
\item $I\Pwr(G):=\left\{K\in\ob(\cat{R}):K\textrm{ is a subhypergraph of }G\right\}$,
\item $\varsigma_{\Pwr(G)}(K):=\check{V}(K)$, $\omega_{\Pwr(G)}(K):=\check{E}(K)$.
\end{enumerate}
Define the incidence hypergraph homomorphism $\xymatrix{G\prod\Pwr(G)\ar[r]^(0.7){\mathrm{elem}_{G}} & \Omega_{\cat{R}}}$ by
\begin{enumerate}
\item $\check{V}\left(\mathrm{elem}_{G}\right)(v,S)=\left\{\begin{array}{cc}
(1,1),	&	v\in S,\\
(0,0),	&	v\not\in S,\\
\end{array}\right.$,
$\check{E}\left(\mathrm{elem}_{G}\right)(e,T)=\left\{\begin{array}{cc}
(1,1),	&	e\in T,\\
(0,0),	&	e\not\in T,\\
\end{array}\right.$,
\item $I\left(\mathrm{elem}_{G}\right)(i,K)=\left\{\begin{array}{ll}
(1,1),	&	i\in I(K),\\
\left(0,(1,1),(1,1)\right),	&	i\not\in I(K),\varsigma_{G}(i)\in\check{V}(K),\omega_{G}(i)\in\check{E}(K),\\
\left(0,(1,1),(0,0)\right),	&	\varsigma_{G}(i)\in\check{V}(K),\omega_{G}(i)\not\in\check{E}(K),\\
\left(0,(0,0),(1,1)\right),	&	\varsigma_{G}(i)\not\in\check{V}(K),\omega_{G}(i)\in\check{E}(K),\\
\left(0,(0,0),(0,0)\right),	&	\varsigma_{G}(i)\not\in\check{V}(K),\omega_{G}(i)\not\in\check{E}(K).\\
\end{array}\right.$
\end{enumerate}
\end{defn}

\begin{theorem}[Power characterization]
Given an incidence hypergraph homomorphism $\xymatrix{G\prod K\ar[r]^(0.6){\phi} & \Omega_{\cat{R}}}$, there is a unique incidence hypergraph homomorphism $\xymatrix{K\ar[r]^(0.4){\hat{\phi}} & \Pwr(G)}$ such that $\mathrm{elem}_{G}\circ \left(G\prod\hat{\phi}\right)=\phi$.
\end{theorem}

\begin{proof}
For $i\in I(K)$, let $T_{i}:=\left\{j\in I(G):I(\phi)(j,i)=(1,1)\right\}$ and define $\xymatrix{K\ar[r]^(0.4){\hat{\phi}} & \Pwr(G)}\in\cat{R}$ by
\begin{itemize}
\item $\check{V}\left(\hat{\phi}\right)(v):=\left\{w\in\check{V}(G):\check{V}(\phi)(w,v)=(1,1)\right\}$,
\item $\check{E}\left(\hat{\phi}\right)(e):=\left\{f\in\check{E}(G):\check{E}(\phi)(f,e)=(1,1)\right\}$,
\item $I\left(\hat{\phi}\right)(i):=\gen_{G}\left(\check{V}\left(\hat{\phi}\right)\left(\varsigma_{K}(i)\right),\check{E}\left(\hat{\phi}\right)\left(\omega_{K}(i)\right),T_{i}\right)$.
\end{itemize}
The proof of the composition condition and uniqueness are routine.
\qed
\end{proof}

\begin{corollary}[Power map]
Let $\xymatrix{G\ar[r]^{\phi} & H}$ be an incidence hypergraph homomorphism.  The power map $\xymatrix{\Pwr(H)\ar[r]^{\Pwr(\phi)} & \Pwr(G)}$ is given by
\begin{enumerate}
\item $\check{V}\Pwr(\phi)(S)=\check{V}(\phi)^{-1}(S)$,
$\check{E}\Pwr(\phi)(T)=\check{E}(\phi)^{-1}(T)$,
\item $I\Pwr(\phi)(K)=\gen_{G}\left(\check{V}(\phi)^{-1}\left(\check{V}(K)\right),\check{E}(\phi)^{-1}\left(\check{E}(K)\right),I(\phi)^{-1}\left(I(K)\right)\right)$.
\end{enumerate}
\end{corollary}

\subsection{Injectivity}

With understanding of the subhypergraphs and partial morphism representers, discussion returns to injectivity.  Using $\tilde{\Box}$, the injective objects of $\cat{R}$ can be readily identified, and manifest much like \cite[Proposition 3.2.1]{Grill2}.  An incidence hypergraph is injective essentially when every edge is incident to every vertex.  To ease the exposition, the following notation is introduced to refer to the set of incidences between a specified vertex and edge.

\begin{defn}
If $G$ is an incidence hypergraph, define $\inc_{G}(v,e):=\varsigma_{G}^{-1}(v)\cap\omega_{G}^{-1}(e)$ for $v\in\check{V}(G)$, $e\in\check{E}(G)$.
\end{defn}

\begin{proposition}[Injective incidence hypergraphs]\label{injective}
An incidence hypergraph $G$ is injective with respect to monomorphisms in $\cat{R}$ if and only if the following conditions hold:
\begin{enumerate}
\item $\check{V}(G)\neq\emptyset$;
$\check{E}(G)\neq\emptyset$;
\item $\inc_{G}(v,e)\neq\emptyset$ for all $v\in\check{V}(G)$ and $e\in\check{E}(G)$.
\end{enumerate}
\end{proposition}

\begin{proof}

$(\Rightarrow)$ As $\eta_{G}$ is monic and $G$ is injective, there is $\xymatrix{\tilde{G}\ar[r]^{\psi} & G}\in\cat{R}$ such that $\psi\circ\eta_{G}=id_{G}$.
\[\xymatrix{
G\\
G\ar[u]^{id_{G}}\ar[r]_{\eta_{G}}	&	\tilde{G}\ar@{..>}[ul]_{\exists\psi}\\
}\]
A calculation shows the following for $v\in\check{V}(G)$ and $e\in\check{E}(G)$:  $\check{V}(\psi)(0,0)\in\check{V}(G)$, $\check{E}(\psi)(0,0)\in\check{E}(G)$, and $I(\psi)\left(0,(1,v),(1,e)\right)\in\inc_{G}(v,e)$.

$(\Leftarrow)$ Fix $u_{0}\in\check{V}(G)$, $g_{0}\in\check{E}(G)$, and $k_{v,e}\in\inc_{G}(v,e)$ for $v\in\check{V}(G)$ and $e\in\check{E}(G)$.  Define $\xymatrix{\tilde{G}\ar[r]^{\psi} & G}\in\cat{R}$ by
\begin{itemize}
\item $\check{V}(\psi)(w):=\left\{\begin{array}{cc}
v,	&	w=(1,v),\\
u_{0},	&	w=(0,0),\\
\end{array}\right.$
$\check{E}(\psi)(f):=\left\{\begin{array}{cc}
e,	&	f=(1,e),\\
g_{0},	&	f=(0,0),\\
\end{array}\right.$
\item $I(\psi)(j):=\left\{\begin{array}{cc}
i,	&	j=(1,i),\\
k_{\left(\check{V}(\psi)\circ\varsigma_{\tilde{G}}\right)(j),\left(\check{E}(\psi)\circ\omega_{\tilde{G}}\right)(j)},	&	\textrm{otherwise}.\\
\end{array}\right.$
\end{itemize}
A calculation shows that $\psi\circ\eta_{G}=id_{G}$, meaning that $G$ is a retract of $\tilde{G}$, and $\tilde{G}$ is injective by \cite[Proposition III.5.6.1]{borceux}.\qed

\end{proof}

The category $\cat{R}$ has enough injectives as $\xymatrix{G\ar[r]^{\eta_{G}} & \tilde{G}}$ is a monomorphism into an injective object for every incidence hypergraph $G$, but this will sadly not be a minimal injective embedding, i.e.\ the injective envelope.  To identify the injective envelope, the essential monomorphisms are characterized as in \cite[Propositions 3.3.1 \& 3.3.2]{Grill2}.  Much like the quiver case, an essential monomorphism only appends vertices, edges, or incidences if none already exist.  By this characterization, $\eta_{G}$ will only be essential in the trivial case when $G=\mathbb{0}_{\cat{R}}$.

\begin{proposition}[Essential monic]\label{essential-mono}
An incidence hypergraph monomorphism $\xymatrix{G\textrm{ }\ar@{>->}[r]^{\phi} & H}$ is essential if and only if the following conditions hold:
\begin{enumerate}
\item if $\check{V}(G)\neq\emptyset$, then $\check{V}(\phi)$ is bijective;
\item if $\check{V}(G)=\emptyset$, then $\card\left(\check{V}(H)\right)\leq 1$;
\item if $\check{E}(G)\neq\emptyset$, then $\check{E}(\phi)$ is bijective;
\item if $\check{E}(G)=\emptyset$, then $\card\left(\check{E}(H)\right)\leq 1$;
\item if $v\in\check{V}(G)$ and $e\in\check{E}(G)$ satisfy $\inc_{G}(v,e)\neq\emptyset$, then
\[
\mathcal{P}I(\phi)\left(\inc_{G}(v,e)\right)
=
\inc_{H}\left(\check{V}(\phi)(v),\check{E}(\phi)(e)\right);
\]
\item if $x\in\check{V}(H)$ and $y\in\check{E}(H)$ satisfy
\[
\left(\left(\varsigma_{H}\circ I(\phi)\right)(i),\left(\omega_{H}\circ I(\phi)\right)(i)\right)\neq(x,y)
\]
for all $i\in I(G)$, then $\card\left(\inc_{H}(x,y)\right)\leq1$.
\end{enumerate}
\end{proposition}

\begin{proof}

$(\Leftarrow)$ Say $\xymatrix{H\ar[r]^{\alpha} & K}\in\cat{R}$ satisfies that $\alpha\circ\phi$ is monic.  Then, all of $\check{V}(\alpha)\circ\check{V}(\phi)$, $\check{E}(\alpha)\circ\check{E}(\phi)$, and $I(\alpha)\circ I(\phi)$ are one-to-one.

If $\check{V}(G)=\emptyset$, then $\card\left(\check{V}(H)\right)\leq 1$, so $\check{V}(\alpha)$ is automatically one-to-one.  If $\check{V}(G)\neq\emptyset$, $\check{V}(\alpha)$ is one-to-one as $\check{V}(\phi)$ is bijective.  By a similar argument, $\check{E}(\alpha)$ is also one-to-one.

Say $i,j\in I(H)$ satisfy that $I(\alpha)(i)=I(\alpha)(j)$.  Let $v:=\varsigma_{H}(i)$ and $e:=\omega_{H}(i)$.  A calculation shows that
\[\begin{array}{ccc}
\check{V}(\alpha)(v)=\check{V}(\alpha)\left(\varsigma_{H}(j)\right)	&	\textrm{and}	&	\check{E}(\alpha)(e)=\check{E}(\alpha)\left(\omega_{H}(j)\right).
\end{array}\]
As $\check{V}(\alpha)$ and $\check{E}(\alpha)$ are one-to-one, $v=\varsigma_{H}(j)$ and $e=\omega_{H}(j)$, giving $i,j\in\inc_{H}(v,e)$.  If there is $k\in I(G)$ such that $i=I(\phi)(k)$, then a calculation shows
\[\begin{array}{ccc}
v=\check{V}(\phi)\left(\varsigma_{G}(k)\right),	&	\textrm{and}	&	e=\check{E}(\phi)\left(\omega_{G}(k)\right),\\
\end{array}\]
which gives that
\[
j
\in\inc_{H}\left(\check{V}(\phi)\left(\varsigma_{G}(k)\right),\check{E}(\phi)\left(\omega_{G}(k)\right)\right)
=\mathcal{P}I(\phi)\left(\inc_{G}\left(\varsigma_{G}(k),\omega_{G}(k)\right)\right).
\]
Then, there is $l\in I(G)$ such that $I(\phi)(l)=j$, so
\[
I(\alpha\circ\phi)(k)
=I(\alpha)\left(i\right)
=I(\alpha)\left(j\right)
=I(\alpha\circ\phi)(l).
\]
As $I(\alpha\circ\phi)$ is one-to-one, $k=l$, giving $i=j$.

Say $i\neq I(\phi)(k)$ for all $k\in I(G)$.  If there was $k\in I(G)$ such that $\left(\left(\varsigma_{H}\circ I(\phi)\right)(k),\left(\omega_{H}\circ I(\phi)\right)(k)\right)=(v,e)$, then a calculation shows
\[
i
\in\inc_{H}\left(\left(\varsigma_{H}\circ I(\phi)\right)(k),\left(\omega_{H}\circ I(\phi)\right)(k)\right)
=\mathcal{P}I(\phi)\left(\inc_{G}\left(\varsigma_{G}(k),\omega_{G}(k)\right)\right),
\]
contradicting that $i\neq I(\phi)(k)$ for all $k\in I(G)$.  Thus, $\left(\left(\varsigma_{H}\circ I(\phi)\right)(k),\left(\omega_{H}\circ I(\phi)\right)(k)\right)\neq(v,e)$ for all $k\in I(G)$.  Thus, $\card\left(\{i,j\}\right)\leq\card\left(\inc_{H}(v,e)\right)\leq1$, so $i=j$.  Therefore, $I(\alpha)$ is one-to-one.

$(\neg\Leftarrow\neg)$ In each case, an appropriate $\xymatrix{H\ar[r]^{\alpha} & K}\in\cat{R}$ is constructed such that $\alpha\circ\phi$ is monic, but $\alpha$ is not monic.
\begin{enumerate}

\item\label{axiom1} Choose $w\in\check{V}(G)$ and $z\in\check{V}(H)\setminus\Ran\left(\check{V}(\phi)\right)$.  Let $\sim$ be the equivalence relation on $\check{V}(H)$ that associates $\check{V}(\phi)(w)$ and $z$, and is equality otherwise.  Let $q:\check{V}(H)\to\check{V}(H)/\sim$ be the quotient map, $K:=\left(\check{V}(H)/\sim,\check{E}(H),I(H),q\circ\varsigma_{H},\omega_{H}\right)$, and $\alpha:=\left(q,id_{\check{E}(H)},id_{I(H)}\right)$.

\item\label{axiom2} Assume $\check{V}(G)=\emptyset$ and $\card\left(\check{V}(H)\right)\geq2$.  Let $x,y\in\check{V}(H)$ satisfy that $x\neq y$.  Let $\sim$ be the equivalence relation on $\check{V}(H)$ that associates $x$ and $y$, and is equality otherwise.  Let $q:\check{V}(H)\to\check{V}(H)/\sim$ be the quotient map, $K:=\left(\check{V}(H)/\sim,\check{E}(H),I(H),q\circ\varsigma_{H},\omega_{H}\right)$, and $\alpha:=\left(q,id_{\check{E}(H)},id_{I(H)}\right)$.

\item This case is dual to case \ref{axiom1}.
\item This case is dual to case \ref{axiom2}.

\item Assume there are $v\in\check{V}(G)$, $e\in\check{E}(G)$, $j\in\inc_{G}(v,e)$, and $z\in\inc_{H}\left(\check{V}(\phi)(v),\check{E}(\phi)(e)\right)\setminus\mathcal{P}I(\phi)\left(\inc_{G}(v,e)\right)$.  Let $\sim$ be the equivalence relation on $I(H)$ that associates $j$ and $z$, and is equality otherwise.  Let $q:I(H)\to I(H)/\sim$ be the quotient map.  Define $\varsigma_{K}:I(H)/\sim\to\check{V}(H)$ and $\omega_{K}:I(H)/\sim\to\check{E}(H)$ by $\varsigma_{K}(q(i)):=\varsigma_{H}(i)$ and $\omega_{K}(q(i)):=\omega_{H}(i)$, which are well-defined by a quick calculation.  Let $K:=\left(\check{V}(H),\check{E}(H),I(H)/\sim,\varsigma_{K},\omega_{K}\right)$ and $\alpha:=\left(id_{\check{V}(H)},id_{\check{E}(H)},q\right)$.

\item Assume that $x\in\check{V}(H)$ and $z\in\check{E}(H)$ satisfy that $(x,z)\neq\left(\left(\varsigma_{H}\circ I(\phi)\right)(i),\left(\omega_{H}\circ I(\phi)\right)(i)\right)$ for all $i\in I(G)$, but $\card\left(\inc_{H}(x,z)\right)\geq2$.  Let $g,h\in\inc_{H}(x,z)$ satisfy that $g\neq h$.  Let $\sim$ be the equivalence relation on $I(H)$ that associates $g$ and $h$, and is equality otherwise.  Let $q:I(H)\to I(H)/\sim$ be the quotient map.  Define $\varsigma_{K}:I(H)/\sim\to\check{V}(H)$ and $\omega_{K}:I(H)/\sim\to\check{E}(H)$ by $\varsigma_{K}(q(i)):=\varsigma_{H}(i)$ and $\omega_{K}(q(i)):=\omega_{H}(i)$, which are well-defined by a quick calculation.  Let $K:=\left(\check{V}(H),\check{E}(H),I(H)/\sim,\varsigma_{K},\omega_{K}\right)$ and $\alpha:=\left(id_{\check{V}(H)},id_{\check{E}(H)},q\right)$.\qed

\end{enumerate}

\end{proof}

\begin{corollary}[Essential $\eta_{G}$]
For an incidence hypergraph $G$, $\eta_{G}$ is essential if and only if $G=\mathbb{0}_{\cat{R}}$.
\end{corollary}

\begin{proof}

$(\Leftarrow)$ A quick check of the conditions in Proposition \ref{essential-mono} proves this case.

$(\neg\Leftarrow\neg)$ If $\check{V}(G)\neq\emptyset$, then $\check{V}\left(\eta_{G}\right)$ is not bijective.  Dually, if $\check{E}(G)\neq\emptyset$, then $\check{E}\left(\eta_{G}\right)$ is not bijective.\qed

\end{proof}

Consequently, the construction of $\tilde{G}$ will be streamlined, much like \cite[Definition 3.3.3]{Grill2}, only adding what is necessary to satisfy the criteria for injectivity.  Equivalently,  this construction uniquely isolates the least injective subhypergraph of $\tilde{G}$ containing the image of $G$.

\begin{defn}[Loading]
Given an incidence hypergraph $G$, define the \emph{loading} of $G$ as the incidence hypergraph $L_{\cat{R}}(G)$ constructed as follows:
\begin{enumerate}
\item $\check{V}L_{\cat{R}}(G):=\left\{\begin{array}{cc}
\check{V}(G),	&	\check{V}(G)\neq\emptyset,\\
\{0\},	&	\check{V}(G)=\emptyset;\\
\end{array}\right.$
$\check{E}L_{\cat{R}}(G):=\left\{\begin{array}{cc}
\check{E}(G),	&	\check{E}(G)\neq\emptyset,\\
\{0\},	&	\check{E}(G)=\emptyset;\\
\end{array}\right.$
\item $IL_{\cat{R}}(G):=\begin{array}{l}
\left(\{1\}\times I(G)\right)
\cup
\left(\{0\}\times\left\{(v,e):\inc_{G}(v,e)=\emptyset\right\}\right)
\end{array},$
\item $\varsigma_{L_{\cat{R}}(G)}(a):=\left\{\begin{array}{cc}
\varsigma_{G}(i),	&	a=(1,i),\\
v,	&	a=(0,v,e),\\
\end{array}\right.$
$\omega_{L_{\cat{R}}(G)}(a):=\left\{\begin{array}{cc}
\omega_{G}(i),	&	a=(1,i),\\
e,	&	a=(0,v,e).\\
\end{array}\right.$
\end{enumerate}
Likewise, define an incidence hypergraph homomorphism $\xymatrix{G\ar[r]^(0.4){j_{G}} & L_{\cat{R}}(G)}$ by $\check{V}\left(j_{G}\right)(v):=v$, $\check{E}\left(j_{G}\right)(e):=e$, and $I\left(j_{G}\right)(i):=(1,i)$.
\end{defn}

\begin{figure}[!ht]
    \centering
    \includegraphics{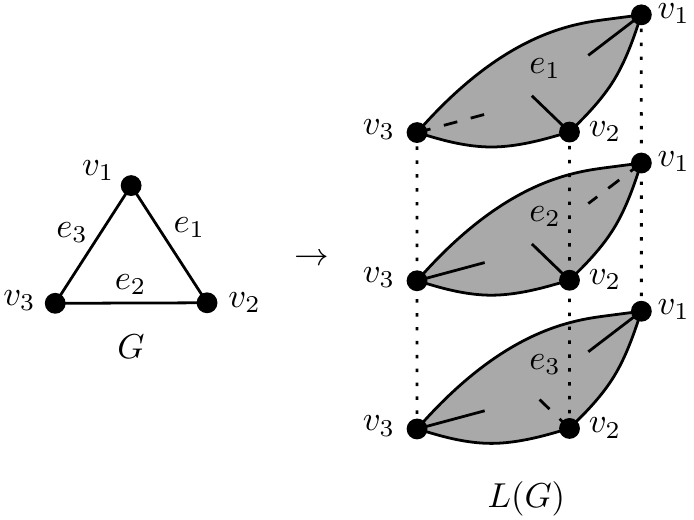}
    \caption{The incidence loading of $K_3$ to produce a uniform hypergraph. New incidences appear dashed within each hyperedge, and the vertices are identified along the dotted vertical lines.}
    \label{fig:LoadingK3}
\end{figure}

\begin{theorem}[Injective envelope]
For an incidence hypergraph $G$, $j_{G}$ is an essential monomorphism, and $L_{\cat{R}}(G)$ is injective with respect to incidence hypergraph monomorphisms.  Thus, $L_{\cat{R}}(G)$ equipped with $j_{G}$ is an injective envelope of $G$.  Moreover, $L_{\cat{R}}(G)$ is isomorphic to the unique minimal injective subhypergraph of $\tilde{G}$ containing the image of $G$ under $\eta_{G}$.
\end{theorem}

\begin{proof}

A quick check shows that $L_{\cat{R}}(G)$ satisfies Proposition \ref{injective}, and that $j_{G}$ satisfies Proposition \ref{essential-mono}.  As $\tilde{G}$ is injective and $j_{G}$ is monic, there is $\xymatrix{L_{\cat{R}}(G)\ar[r]^(0.6){\psi} & \tilde{G}}\in\cat{R}$ such that $\psi\circ j_{G}=\eta_{G}$.
\[\xymatrix{
\tilde{G}\\
G\textrm{ }\ar@{>->}[r]_{j_{G}}\ar[u]^{\eta_{G}}    &   L_{\cat{R}}(G)\ar@{..>}[ul]_{\exists\psi}\\
}\]
As $j_{G}$ is essential monic and $\eta_{G}$ is monic, $\psi$ is monic.  Thus, $L_{\cat{R}}(G)$ equipped with $\psi$ is a subobject of $\tilde{G}$.  By Theorem \ref{subobject}, $L_{\cat{R}}(G)$ with $\psi$ corresponds to a subhypergraph of $\tilde{G}$ via isomorphism.  For $v\in\check{V}(G)$, $e\in\check{E}(G)$, and $i\in I(G)$, one has
\begin{itemize}
\item $(1,v)=\check{V}\left(\eta_{G}\right)(v)=\check{V}\left(\psi\circ j_{G}\right)(v)=\check{V}\left(\psi\right)\left(\check{V}\left(j_{G}\right)(v)\right)=\check{V}\left(\psi\right)\left(v\right)\in\check{V}(\psi)\left(\check{V}\left(L_{\cat{R}}(G)\right)\right)$,
\item $(1,e)=\check{E}\left(\eta_{G}\right)(e)=\check{E}\left(\psi\circ j_{G}\right)(e)=\check{E}\left(\psi\right)\left(\check{E}\left(j_{G}\right)(e)\right)=\check{E}\left(\psi\right)\left(e\right)\in\check{E}(\psi)\left(\check{E}\left(L_{\cat{R}}(G)\right)\right)$,
\item $(1,i)=I\left(\eta_{G}\right)(i)=I\left(\psi\circ j_{G}\right)(i)=I\left(\psi\right)\left(I\left(j_{G}\right)(i)\right)=I\left(\psi\right)\left(1,i\right)\in I(\psi)\left(I\left(L_{\cat{R}}(G)\right)\right)$.
\end{itemize}
Hence, the image of $G$ under $\eta_{G}$ is contained within the image of $L_{\cat{R}}(G)$ under $\psi$.  If $\check{V}(G)=\emptyset$, then $\check{V}(\psi)(0)=(0,0)$.  Dually, $\check{E}(\psi)(0)=(0,0)$ if $\check{E}(G)=\emptyset$.  If $v\in\check{V}L_{\cat{R}}(G)$ and $e\in\check{E}L_{\cat{R}}(G)$ satisfy that $\inc_{G}(v,e)=\emptyset$, then,
\begin{itemize}
\item $\varsigma_{\tilde{G}}\left(I(\psi)\left(0,v,e\right)\right)=\check{V}(\psi)\left(\varsigma_{L_{\cat{R}}(G)}\left(0,v,e\right)\right)=\check{V}(\psi)(v)$,
\item $\omega_{\tilde{G}}\left(I(\psi)\left(0,v,e\right)\right)=\check{E}(\psi)\left(\omega_{L_{\cat{R}}(G)}\left(0,v,e\right)\right)=\check{E}(\psi)(e)$.
\end{itemize}
A calculation shows $\inc_{\tilde{G}}\left(\check{V}(\psi)(v),\check{E}(\psi)(e)\right)=\left\{\left(0,\check{V}(\psi)(v),\check{E}(\psi)(e)\right)\right\}$, which gives that $I(\psi)\left(0,v,e\right)=\left(0,\check{V}(\psi)(v),\check{E}(\psi)(e)\right)$.  Thus, $\psi$ and, consequently, its image are uniquely determined.\qed

\end{proof}

%%%%%%%%%%%%%
%%%%%%%%%%%%%
\section{Oriented Hypergraphic Total Minor Polynomials}
\label{sec:MTT}

\subsection{General Coefficient Theorems}

We demonstrate an oriented hypergraphic generalization of Chaiken's all-minors matrix-tree theorem \cite{Seth1} to all integer matrices using the injective envelope of the underlying incidence hypergraph, and the sign/monomial pair in the total minor polynomial generalizes Sachs' theorem \cite{Sim1, SGBook}. This is a strengthening of the results of \cite{OHSachs} while simultaneously providing insight on the connection between the boolean ideals of graph contributors and Tutte's arborescence theorem discussed in \cite{OHMTT}. On the other hand, a natural unifying theorem is unlikely to exist for the traditional theory of set systems and simple graphs as they lack two key qualities:  (1) mapping within an edge (i.e.\ locally graphic behavior) and (2) an injective envelope that forms a uniform hypergraph.  The locally  graphic property was key to the characterization given in \cite{AH1}, and the uniform hypergraph structure will be used in the present work below. Since $\cat{R}$ possesses a subobject classifier define $\delta_G(H)$ to be the $G$-subobject indicator that is $1$ if $H$ is a subobject of $G$ and $0$ otherwise. The \emph{$0$-loading of an oriented hypergraph $G$} is the oriented hypergraph $L^{0}(G)$ that is obtained by taking the loading of the underlying incidence hypergraph and extending the orientation function $\sigma$ to $\sigma_{0}$ where $\sigma_{0}\big|_{I_0} = 0$, where $I_0$ is the set of newly created incidence in the loading. 

Let $U, W \subseteq V$ such that $|U| = |W|$, and consider two total orderings of $U$ and $W$, denoted $\mathbf{u}$ and $\mathbf{w}$. The map $u_i \rightarrow w_i$ between these total orderings forms the \emph{$[\mathbf{u},\mathbf{w}]$-equivalence-class} of contributors, let ${\mathcal{C}}(G;\mathbf{u},\mathbf{w})$ be the set of contributors in $G$ where $c(u_i)=w_i$. Let $\widehat{\mathcal{C}}(G;\mathbf{u},\mathbf{w})$ be the set obtained by removing the $\mathbf{u} \rightarrow \mathbf{w}$ mappings from ${\mathcal{C}}(G;\mathbf{u},\mathbf{w})$, the elements of $\widehat{\mathcal{C}}(G;\mathbf{u},\mathbf{w})$ are called the \emph{reduced} $[\mathbf{u},\mathbf{w}]$-equivalent contributors. It is important to note some $[\mathbf{u},\mathbf{w}]$-equivalency classes may be empty for a given oriented hypergraph $G$, this is rectified in $L^{0}(G)$ where no class is empty, and the non-zero contributors correspond to the evaluations of the subobject indicator.

\begin{figure}[ht!]
    \centering
    \includegraphics[scale=1]{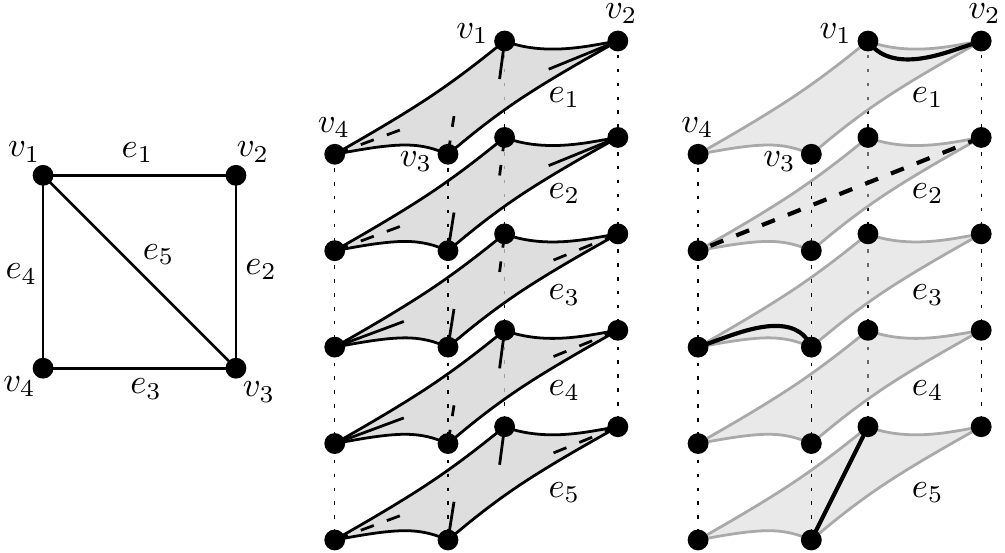}
    \caption{Left: A graph $G$; Middle: The loading of $G$ ($0$-incidences appear dashed); Right: An example where the contributor for permutation $(1243)$ only exists in the loading.}
    \label{fig:4vert}
\end{figure}

Reconstructing a reduced $[\mathbf{u},\mathbf{w}]$-equivalent contributor to a contributor, while not unique, always produces a contributor associated to the same permutation.

\begin{lemma}
\label{PermLemma}
For each $c \in \widehat{\mathcal{C}}(G;\mathbf{u},\mathbf{w})$, the set of all $\check{c} \in \mathcal{C}(G;\mathbf{u},\mathbf{w})$ formed by reintroducing $\mathbf{u} \rightarrow \mathbf{w}$ to $c$ are permutomorphic.
\end{lemma}

Let $\mathbf{X}$ be the $V \times V$ matrix whose $ij$-entry is $x_{ij}$. Let $\chi ^{D}(\mathbf{M},\mathbf{x}):=\det (\mathbf{X-M)}$ be the determinant-based multivariable characteristic polynomial and $\chi ^{P}(%
\mathbf{M},x):=\mathrm{perm}(\mathbf{X-M)}$ be the permanent-based multivariable characteristic polynomial. Also, let $ec(c)$, $oc(c)$, $pc(c)$ and $nc(c)$ be the number of even, odd, positive, and negative components in a (sub-)contributor $c$, respectively. While $bs(c)$ denotes the number of backsteps in contributor $c$. It is worth noting that backsteps are technically negative weak walks that do not arise from adjacencies, but we choose to leave the count separate to illustrate the difference between Laplacian and adjacency matrix formulations.

\begin{theorem}[Total-minor Polynomial]
\label{t:Main}
Let $G$ be an oriented hypergraph with adjacency matrix $\mathbf{A}_{G}$ and Laplacian matrix $\mathbf{L}_{G}$, then 

\begin{enumerate}
\item $\chi ^{P}(\mathbf{A}_{G},\mathbf{x})=\dsum\limits_{ [\mathbf{u},\mathbf{w}]}
\left(
\dsum\limits_{\substack{s \in \widehat{\mathcal{S}}(L^{0}(G);\mathbf{u},\mathbf{w}) \\ \sgn(s) \neq 0}}
(-1)^{oc(s)+nc(s)}
\right)
\dprod\limits_{i}x_{u_{i},w_{i}}$,

\item $\chi ^{D}(\mathbf{A}_{G},\mathbf{x})=\dsum\limits_{ [\mathbf{u},\mathbf{w}]}
\left(
\dsum\limits_{\substack{s \in \widehat{\mathcal{S}}(L^{0}(G);\mathbf{u},\mathbf{w}) \\ \sgn(s) \neq 0}}
(-1)^{ec(\check{s})+oc(s)+nc(s)}
\right)
\dprod\limits_{i}x_{u_{i},w_{i}}$,

\item $\chi ^{P}(\mathbf{L}_{G},\mathbf{x})=\dsum\limits_{ [\mathbf{u},\mathbf{w}]}
\left(
\dsum\limits_{\substack{c \in \widehat{\mathcal{C}}(L^{0}(G);\mathbf{u},\mathbf{w}) \\ \sgn(c) \neq 0}}
(-1)^{nc(c)+bs(c)}
\right)
\dprod\limits_{i}x_{u_{i},w_{i}}$,

\item $\chi ^{D}(\mathbf{L}_{G},\mathbf{x})=\dsum\limits_{ [\mathbf{u},\mathbf{w}]}
\left(
\dsum\limits_{\substack{c \in \widehat{\mathcal{C}}(L^{0}(G);\mathbf{u},\mathbf{w}) \\ \sgn(c) \neq 0}}
(-1)^{ec(\check{c})+nc(c)+bs(c)}
\right)
\dprod\limits_{i}x_{u_{i},w_{i}}$.
\end{enumerate}
\end{theorem}

\begin{proof}
The first half of the proof is an adaptation of the author's work in \cite[Theorem 4.2.1]{OHSachs}, before utilizing the injective closure and the zero-loading of the incidence hypergraph.

Let $p : \overrightarrow{P}_1 \rightarrow G$, and let $q$ denote an incidence-monic map from $\overrightarrow{P}_1 \rightarrow G$. For a given permutation $\pi \in S_{V}$, let $\mathcal{P}_{\pi} = \{p \mid p(t)=v \text{ and } p(h)=\pi(v)\}$, and $\mathcal{Q}_{\pi}$ be defined similarly for incidence monic maps.

\textit{Proof of 1.} 
For a given permutation $\pi$ and vertex $v$ let $\alpha :v\rightarrow \left\{ x_{v,\pi(v)},-\dsum\limits_{q \in \mathcal{Q}_{\pi}}\sgn(q(\overrightarrow{P}_{1})) \right\}$ be the function that chooses either the variable or the value at coordinate $(v,\pi(v))$. Let $\mathcal{A}_{\pi }$ be the set of all $\alpha$ for a given $\pi$.

Thus, $\chi ^{P}(\mathbf{A}_{G},\mathbf{x})$ can be written as

\begin{eqnarray*}
\chi ^{P}(\mathbf{A}_{G},\mathbf{x}) &=&\mathrm{perm}(\mathbf{X-A}_{G}) \\
&=&\sum\limits_{\pi \in S_{V}}\prod\limits_{v\in V}\sum\limits_{\alpha \in 
\mathcal{A}_{\pi }}\alpha (v)\text{.}
\end{eqnarray*}
Distributing we get
\begin{equation*}
=\sum\limits_{\pi \in S_{V}}\sum\limits_{\beta \in \mathcal{B}_{\pi
}}\prod\limits_{v\in V}\beta (v)\text{,}
\end{equation*}
where $\mathcal{B}_{\pi }$ is the set of all functions $\beta :V\rightarrow \left\{ x_{v,\pi(v)},
-\dsum\limits_{q \in \mathcal{Q}_{\pi}}\sgn(q(\overrightarrow{P}_{1}))\right\} $. This can be recognized as passing to the $\Set$ exponential. For each $\beta \in \mathcal{B}_{\pi }$ let $U_{\beta }\subseteq V$ 
be the set of vertices mapped to an $x_{v,\pi(v)}$.

This gives:
\begin{eqnarray*}
=\sum\limits_{\pi \in S_{V}}\sum\limits_{\beta \in 
\mathcal{B}_{\pi }}\left[ \left(\prod\limits_{u\in\overline{U}_{\beta }}\beta (v)\right)  
\prod\limits_{u\in {U}_{\beta }}x_{u,\pi(u)}
\right] \text{.}
\end{eqnarray*}
Evaluating $\beta (v)$ we have:
\begin{eqnarray*}
=\sum\limits_{\pi \in S_{V}}\sum\limits_{\beta \in 
\mathcal{B}_{\pi }}\left[ \left(\prod\limits_{u\in\overline{U}_{\beta }}\dsum\limits_{q \in \mathcal{Q}_{\pi}(G|\overline{U}_{\beta})}-\sgn(q(\overrightarrow{P}_{1}))\right)  
\prod\limits_{u\in {U}_{\beta }}x_{u,\pi(u)}
\right] \text{.}
\end{eqnarray*}
Where $\mathcal{Q}_{\pi}(G|\overline{U}_{\beta})$ is the set of maps $q$ whose tail-set is $\overline{U}_{\beta}$ and head-set is $\pi(\overline{U}_{\beta})$.  Distributing again produces:
\begin{eqnarray*}
=\sum\limits_{\pi \in S_{V}}
\sum\limits_{U\subseteq V }
\left[
\sum\limits_{s \in {\mathcal{S}}_{\pi }(G|\overline{U}))}
\left(
\prod\limits_{u\in\overline{U}} \sigma (s(i_{v}))\sigma (s((j_{v})) \right)
\right]
\prod\limits_{u\in {U}}x_{u,\pi(u)}  \text{,}
\end{eqnarray*}
where $\mathcal{S}_\pi(G|\overline{U})$ is the restricted set of strong contributors that correspond to permutation $\pi$ with tails at $\overline{U}$.

Now pass to the injective envelope of the underlying incidence hypergraph and extend the incidence orientation function $\sigma$ to $\sigma_{L}$ 
such that $\sigma_{L}(i)=\sigma(i)$ for all $i \in I(G)$ and the new incidence orientations are assigned arbitrary. Using the $G$-subobject indicator $\delta_G$ the sum can be rewritten as:
\begin{eqnarray*}
=\sum\limits_{\pi \in S_{V}}
\sum\limits_{U\subseteq V }
\left[
\sum\limits_{s \in {\mathcal{S}}_{\pi }(L(G))}
\delta_G(s|\overline{U}) \left(
\prod\limits_{u\in\overline{U}} \sigma_{L} (s(i_{v}))\sigma_{L} (s((j_{v})) \right)
\right]
\prod\limits_{u\in {U}}x_{u,\pi(u)}  \text{.}
\end{eqnarray*}

The product of signs is evaluated by first factoring out a negative for each adjacency producing a value of $(-1)^{oc(s)}$, and then factoring out a negative for each negative adjacency producing a value of $(-1)^{nc(s)}$ --- leaving behind only $+1$'s for all adjacencies, and reducing to a count of subcontributors of the underlying incidence hypergraph,
\begin{eqnarray*}
=\sum\limits_{\pi \in S_{V}}
\sum\limits_{ U\subseteq V }
\left[
\sum\limits_{s \in {\mathcal{S}}_{\pi }(L(G))}
\delta_G(s|\overline{U}) \cdot (-1)^{oc(s)+nc(s)}
\right]
\prod\limits_{u\in {U}}x_{u,\pi(u)} \text{.}
\end{eqnarray*}
Resolving $\delta_G$ and letting $w_i = \pi(u_i)$, we pass to the $0$-loading $L^{0}(G)$ of the oriented hypergraph and combine the first two sums.
\begin{eqnarray*}
=\sum\limits_{ [\mathbf{u},\mathbf{w}]}
\left(
\sum\limits_{\substack{s \in \widehat{\mathcal{S}}(L^{0}(G);\mathbf{u},\mathbf{w}) \\ \sgn(s) \neq 0}}
(-1)^{oc(s)+nc(s)}
\right)
\prod\limits_{i}x_{u_{i},w_{i}}  \text{.}
\end{eqnarray*}

\textit{Proof of 2.}
Proceeding as in part 1 with the inclusion of the sign of the permutation we get
\begin{eqnarray*}
\chi ^{D}(\mathbf{A}_{G},\mathbf{x}) &=&\det(\mathbf{X-A}_{G}) \\
&=&
\sum\limits_{\pi \in S_{V}} \epsilon(\pi) \sum\limits_{ U\subseteq V } 
\left[
\sum\limits_{s \in \mathcal{S}_{\pi }(L(G))}
\delta_G(s|\overline{U}) \cdot
(-1)^{oc(s)+nc(s)}
\right]
\prod\limits_{u\in {U}} x_{u,\pi(u)}  \text{.}
\end{eqnarray*}

Using the fact that the sign of a permutation is equal to $(-1)^{ec(\pi)}$, where $ec(\pi)$ is the number of even algebraic cycles in $\pi$, and each contributor is associated to a unique permutation we have
\begin{eqnarray*}
= \sum\limits_{\pi \in S_{V}} \sum\limits_{ U\subseteq V }
\left[
\sum\limits_{s \in \mathcal{S}_{\pi }(L(G))} (-1)^{ec(\check{s})} \cdot
\delta_G(s|\overline{U}) \cdot 
(-1)^{oc(s)+nc(s)}
\right]
\prod\limits_{u\in {U}} x_{u,\pi(u)}  \text{.}
\end{eqnarray*}

Again, resolve $\delta_G$ but this time observe that the $(-1)^{oc(s)+nc(s)}$ values are for subcontributors where $U \rightarrow \pi(U)$ is removed, while the value $(-1)^{ec(s)}$ remains unchanged as it is determined by a permutation. Let $\check{s}$ be any maximal contributor obtained by extending the subcontributor $s$ by $U \rightarrow \pi(U)$, all such contributors are permutomorphic by Lemma \ref{PermLemma}.
\begin{eqnarray*}
=\sum\limits_{ [\mathbf{u},\mathbf{w}]}
\left(
\sum\limits_{\substack{s \in \widehat{\mathcal{S}}(L^{0}(G);\mathbf{u},\mathbf{w}) \\ \sgn(s) \neq 0}}
(-1)^{ec(\check{s})+oc(s)+nc(s)}
\right)
\prod\limits_{i}x_{u_{i},w_{i}}  \text{.}
\end{eqnarray*}

\textit{Proofs of 3. and 4.}
The proofs for the Laplacian are similar with the following modifications: (1) switch from incidence-monic maps $\mathcal{Q}_{\pi}$ to arbitrary maps $\mathcal{P}_{\pi}$ to allow backsteps and sum over contributors instead of strong contributors; (2) since $\mathbf{L}_{G} = \mathbf{D}_{G} -\mathbf{A}_{G}$ there is no need to factor out a $-1$ for each adjacency, and instead factor out a $-1$ for each backstep.\qed
\end{proof}

\subsubsection{Examples}

\begin{example}
Consider the $K_3$ example from Subsection \ref{ssec:MTTandSach}. This is the incidence graph $G_1$ from Figure \ref{fig:K3Graph} with the $16$ contributors appearing on the left of Figure \ref{Fig:Hypergraph3}. We know $\det(x\mathbf{I}-\mathbf{A})=x^3-3x-2$ and $\perm(x\mathbf{I}-\mathbf{A})=x^3+3x-2$. There are two strong contributors, namely the two $3$-cycles. These provide a constant of $-2$ in the adjacency matrix for both the determinant and permanent characteristic polynomials as each has no isolated vertices, both have odd parity, and neither is negative --- thus each has a value of $(-1)^{0+1}$. Moreover, the largest magnitude the constant could be is $2$ as there are two strong contributors. The sign difference on the $x$ term is clear from the even parity inclusion from Theorem \ref{t:Main}.

For the Laplacian we have $\det(x\mathbf{I}-\mathbf{L})=x^3-6x^2+9x$. The maximum magnitude for the Laplacian constant is $16$, the number of contributors. The actual Laplacian constant term is $0$ as the contributors fall into alternating signed Boolean lattices and sum to $0$; see Figure \ref{Fig:Hypergraph3} and \cite{OHMTT} for more details.
\end{example}

\begin{example}
We now expand on our $K_3$ example to determine $\chi ^{P}(\mathbf{A}_{G_1},\mathbf{x})$. The constant term will still be produced by the two $3$-cycle strong contributors, however, in $\chi ^{P}(\mathbf{A}_{G_1},\mathbf{x})$, the subcontributors also contribute additional monomials shown in Figure \ref{fig:K3SubCont}.

\begin{figure}[ht!]
    \centering
    \includegraphics{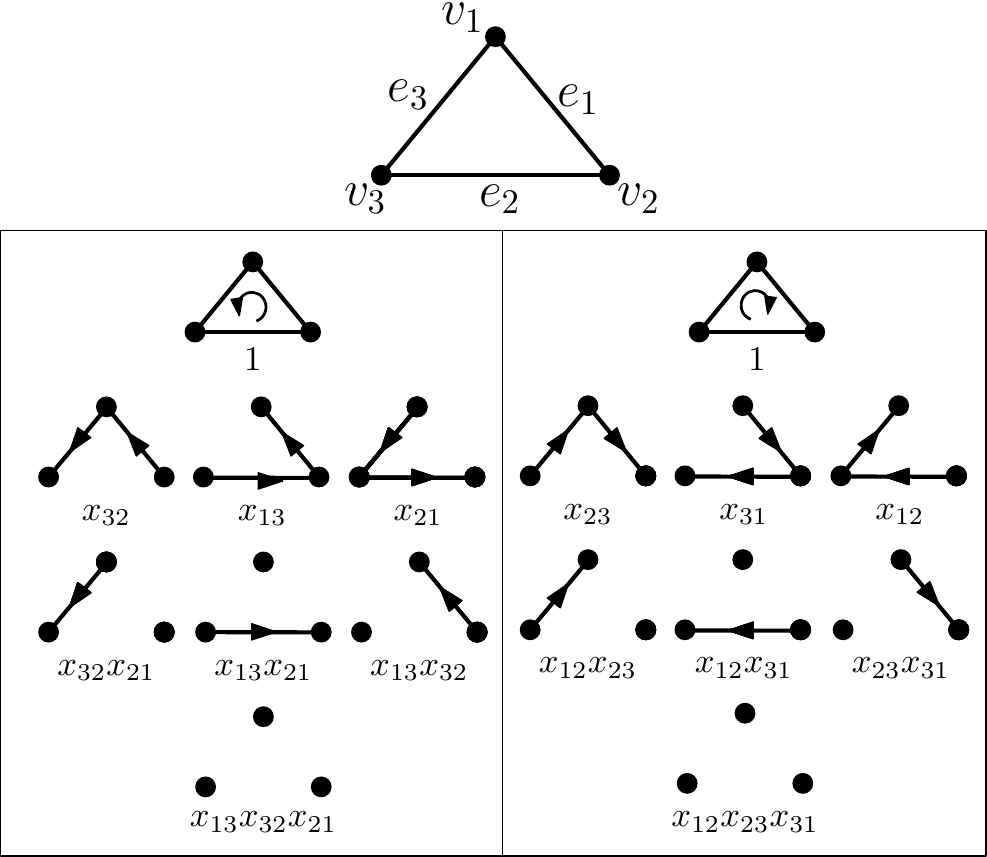}
    \caption{The two non-zero strong contributors of $K_3$ and their subcontributors ordered by monomial division.}
    \label{fig:K3SubCont}
\end{figure}
 
Thus, $\chi ^{P}(\mathbf{A}_{G_1},\mathbf{x})$ would contain the following expression resulting from the strong contributor resulting from permutation $(123)$:
$$
x_{12}x_{23}x_{31}-x_{12}x_{23}-x_{12}x_{31}-x_{23}x_{31}+x_{12}+x_{23}+x_{31}-1
$$
where the coefficients are determined by the formula $(-1)^{oc(s)+nc(s)}$ for each subcontributor. Here, there are no negative edges, so the sign is determined by the odd parity only. Note, that in the determinant case the value $(-1)^{ec(\check{s})}$ is determined by the maximal strong contributor corresponding to the constant coefficient. Since Figure \ref{fig:K3SubCont} contains all the restrictions of the strong subcontributors, the maximum magnitude of the coefficients of the adjacency matrix of a signed $K_3$ are $1$ for the monomials listed, with the exception of $2$ for the constant.
\end{example}

\begin{example}
Consider the oriented hypergraph $(G_2,\sigma_2)$ from Figure \ref{fig:2Orientations} with contributors listed on the right of Figure \ref{Fig:Hypergraph3}. We have
\begin{eqnarray*}
\chi ^{D}(\mathbf{L}_{(G_2,\sigma_2)},\mathbf{x})&=&\det{(\mathbf{X}-\mathbf{L}_{(G_2,\sigma_2)})}=\det\left[
\begin{array}{ccc}
x_{11}-1 & x_{12}-1 & x_{13}+1 \\ 
x_{21}-1 & x_{22}-1 & x_{23}+1 \\ 
x_{31}+1 & x_{32}+1 & x_{33}-1 %
\end{array}
\right]\\
&=&x_{11}x_{22}x_{33}-x_{11}x_{23}x_{32}-x_{13}x_{22}x_{31}-x_{12}x_{21}x_{33}+x_{12}x_{23}x_{31}+x_{13}x_{21}x_{32}\\
&-&x_{11}x_{22}-x_{11}x_{23}-x_{11}x_{32}-x_{11}x_{33}-x_{13}x_{22}-x_{22}x_{31}-x_{22}x_{33}-x_{23}x_{31}-x_{13}x_{32}\\
&+&x_{12}x_{21}+x_{13}x_{21}+x_{12}x_{23}+x_{12}x_{31}+x_{13}x_{31}+x_{21}x_{32}+x_{23}x_{32}+x_{12}x_{33}+x_{12}x_{33}
\end{eqnarray*}
where the constant and linear terms all have coefficient zero. The non-zero coefficients are obtained via restricted contributors from Figure \ref{Fig:Hypergraph3} and the signing from Theorem \ref{t:Main}
\end{example}

\subsection{Bidirected Graphs and  k-arborescences}

Building on the work in \cite{OHMTT}, we group contributors of bidirected graphs into Boolean activation classes, and show the single-element classes for a given degree-$k$ monomial are in one-to-one correspondence with Tutte's $k$-arborescences. Moreover, the remaining elements in the activation class provide an upper bound on absolute value of the coefficient for the associated monomial.

First we collect the relevant definitions from \cite{OHMTT}. A \emph{pre-contributor of }$\emph{G}$ is an incidence preserving function $p:\coprod\limits_{v\in V}\overrightarrow{P}_{1}\rightarrow G$ with $p(t_{v})=v$. For a pre-contributor $p$ with $p(t_{v})\neq p(h_{v})$, define \emph{packing a directed adjacency of a pre-contributor }$p$\emph{\ into
a backstep at vertex }$v$ to be a pre-contributor $p_{v}$ such that $p_{v}=p$
for all $u\in V\smallsetminus v$, and for vertex $v$ 
\begin{eqnarray*}
p((\overrightarrow{P}_{1})_{v}) &=&(v,i,e,j,w)\text{, }i\neq j\text{,} \\
\text{and }p_{v}((\overrightarrow{P}_{1})_{v}) &=&(v,i,e,i,v)\text{.}
\end{eqnarray*}%
Thus, the head-incidence and head-vertex of adjacency $p((\overrightarrow{P}%
_{1})_{v})$ are identified to the tail-incidence and tail-vertex. It is clear that this is equivalent to \emph{tail-equivalence}. However, since all edges in a bidirected graph have size equal to $2$ there is a unique target for the head vertex to map to. \emph{Unpacking a backstep of a pre-contributor }$p$\emph{\ into an
adjacency out of vertex }$v $\emph{\ }is a pre-contributor $p^{v}$ defined analogously where, for vertex $v$ the head-incidence and head-vertex of backstep $p((\overrightarrow{P}_{1})_{v})$ are identified to the unique incidence and vertex that would complete the adjacency in bidirected graph $G$. \emph{Activating
a circle of contributor }$c$ is a minimal sequence of unpackings that results in a new contributor, and define the \emph{activation partial order} $\leq _{a}$ where $c\leq _{a}d$ if $d$ is formed by a sequence of activations starting with $c$. This induces the \emph{activation equivalence relation} $\sim _{a}$where $c\sim _{a}d$ if $c\leq _{a}d$ or $d\leq _{a}c$, and the elements of $\mathcal{C}(G)/\sim _{a}$ are called the \emph{activation classes of }$G$. Let $\mathcal{A}(\mathbf{u};\mathbf{w};G)$ denote the $[\mathbf{u}, \mathbf{w}]$-equivalent elements in activation class $\mathcal{A}$, and let $\hat{\mathcal{A}}(\mathbf{u};\mathbf{w};G)$ be the elements of $\mathcal{A}(\mathbf{u};\mathbf{w};G)$ with the adjacency or backstep from $u_{i}$ to $w_{i}$ is removed for each $i$. 

\begin{lemma}[\cite{OHMTT}, Lemma 3.6]
\label{boolean} For a bidirected graph $G$, all activation classes of $G$ are Boolean lattices.
\end{lemma}

\begin{lemma}[\cite{OHMTT}, Theorem 3.11]
The elements of $\mathcal{A}(\mathbf{u};\mathbf{w};G)$ form a sub-Boolean lattice of $\mathcal{A}$ determined by sequential order ideals.
\end{lemma}

\begin{lemma}[\cite{OHMTT}, Lemma 4.5]
If $G$ is a bidirected graph, then the set of elements in all single-element $\hat{\mathcal{A}}_{\neq 0}(u;w;G^{\prime })$ is unpacking equivalent
to the set of spanning trees of $G$. Where $G'$ is the injective envelope of $G$ in the category of graphs (i.e. the completion of the underlying graph).
\label{lemspan}
\end{lemma}

The total minor polynomials can be used to extend the results of Lemma \ref{lemspan}. 

\begin{theorem}
\label{karbor}
In a bidirected graph $G$ the set of all elements in single-element $\hat{\mathcal{A}}_{\neq 0}(\mathbf{u};\mathbf{w};L(G))$ is unpacking equivalent to $k$-arborescences. Moreover, the $i^{th}$ component in the arborescence has sink $u_i$, and the vertices of each component are determined by the linking induced by $c^{-1}$ between all $u_i \in U \cap \overline{W} \rightarrow \overline{U}$ or unpack into a vertex of a linking component.
\end{theorem}

\begin{proof}
Let $\hat{\mathcal{A}}_{\neq 0}(\mathbf{u};\mathbf{w};L(G))$ contain a single element contributor, call it $c$. If $c$ contains a circle, then there would be a $(\mathbf{u}, \mathbf{w})$-equivalent contributor $d$ with $d <_a c$ such that there is a sequence of unpackings that activates into $c$, and $\hat{\mathcal{A}}_{\neq 0}(\mathbf{u};\mathbf{w};L(G))$ would contain more than one element. Moreover, $c$ cannot have any circle that can be activated, or there would be $(\mathbf{u}, \mathbf{w})$-equivalent contributor $d'$ with $c <_a d'$, and $\hat{\mathcal{A}}_{\neq 0}(\mathbf{u};\mathbf{w};L(G))$ would contain more than one element.

Additionally, since the single-element of 
$\hat{\mathcal{A}}_{\neq 0}(\mathbf{u};\mathbf{w};L(G))$ is a non-zero contributor in $L(G)$, the corresponding totally unpacked pre-contributor $p$ exists in $G$. Thus, $p$ is circle-free with exactly $|V|$ vertices and $|V|-k$ edges, so it is a $k$-arborescence. 

By the Linking Lemma every $U \rightarrow W$ matching has an induced linking in the opposite direction. Let $u_i \in U$. If $u_i \notin \overline{W}$, then both the entrant and salient edges are missing at $u_i$, and $u_i$ is isolated before unpacking. If $u_i \in \overline{W}$, then only the salient edge is missing at $u_i$. Since all remaining vertices can only posses backsteps that unpack towards a vertex in the connected component containing a $u_i$, each $u_i$ is the sink of an inward-arborescence. Additionally, all vertices either are in the induced linking or unpack into one of the components.
\qed
\end{proof}

\subsubsection{Example}

\begin{example}
To determine the coefficient for $x_{12}x_{23}$ in $\chi^{D}(\mathbf{L}_{G},\mathbf{x})$ for the graph (i.e. all edges positive) in Figure \ref{fig:4vert} observe that the set $U=\{1,2\}$ corresponding to all first subscript entries and the set $W=\{2,3\}$ corresponding to all second subscript entries. The $[(1,2),(2,3)]$-equivalent contributors, their non-zero reduced contributors, and the unpacking into an inward arborescence appear in Figure \ref{fig:ArborEx}. Each component in each arborescence has an element of $U$ as a sink as well as the corresponding linking in the reduced contributor. The remaining backsteps unpack into the linkings; hence, towards the sinks.

\begin{figure}[ht!]
    \centering
    \includegraphics{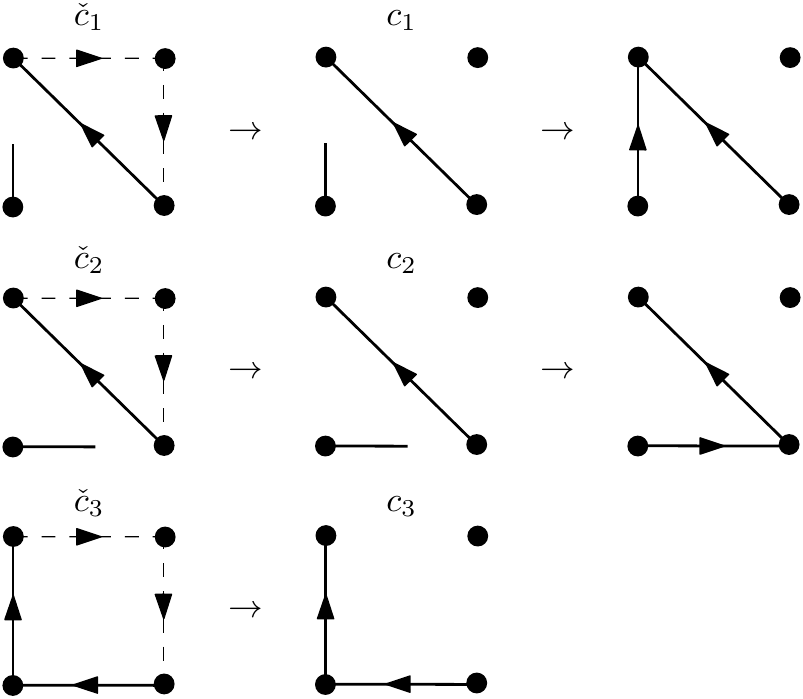}
    \caption{The three $[(1,2),(2,3)]$-equivalent contributors, their reduced subcontributor in $G$ with linking, and the unpacked inward arborescence rooted at $v_1$.}
    \label{fig:ArborEx}
\end{figure}

The signing function for the Laplacian determinant is $(-1)^{ec(\check{c})+nc(c)+bs(c)}$, where $\sgn(c_1)=\sgn(c_2)=(-1)^{0+0+1}=-1$ while $\sgn(c_3)=(-1)^{1+0+0}=-1$, thus the coefficient of $x_{12}x_{23}$ in $\chi^{D}(\mathbf{L}_{G},\mathbf{x})$ is $-3$. Similarly, the coefficient of $x_{12}x_{23}$ in $\chi^{P}(\mathbf{L}_{G},\mathbf{x})$ is $-1$ using the signing function $(-1)^{nc(c)+bs(c)}$.
\end{example}

\section*{Acknowledgments}
The authors would like to thank the anonymous referee for their time and for the helpful and motivating suggestions to improve the quality and delivery of the paper.

%%%%%%%%%%%%%%%%%%%%%%%%%%%%%%%%
\newpage
%\section*{References}
\providecommand{\bysame}{\leavevmode\hbox to3em{\hrulefill}\thinspace}
\providecommand{\MR}{\relax\ifhmode\unskip\space\fi MR }
\providecommand{\MRhref}[2]{%
  \href{http://www.ams.org/mathscinet-getitem?mr=#1}{#2}
}
\providecommand{\href}[2]{#2}

\end{document}